\documentclass[a4paper,twoside,11pt]{article}
\usepackage{a4wide}
\usepackage[utf8]{inputenc}
\usepackage[english]{babel} 
\usepackage{amsmath,amssymb,amsthm,amsfonts}
\usepackage{mathrsfs}
\usepackage{bbm}
\usepackage{graphicx}
\usepackage[colorlinks=true,citecolor=red,linkcolor=blue,urlcolor=blue,pdfstartview=FitH]{hyperref}
\usepackage{fancyhdr}
\usepackage{appendix}
\usepackage{vruler}

\usepackage{enumitem}

\theoremstyle{plain}
\newtheorem{theorem}{Theorem}[section]
\newtheorem{lemma}[theorem]{Lemma}
\newtheorem{corollary}[theorem]{Corollary}
\newtheorem{proposition}[theorem]{Proposition}
\newtheorem{fact}[theorem]{Fact}

\theoremstyle{definition}

\theoremstyle{remark}
\newtheorem{remark}[theorem]{Remark}

\numberwithin{equation}{section}

\def\C{\ensuremath{\mathbb{C}}}

\def\D{\ensuremath{\mathbb{D}}}
\def\Db{\ensuremath{\overline{\D}}}

\def\e{\ensuremath{\mathrm{e}}}
\def\E{\ensuremath{\mathbb{E}}}
\def\Et{\ensuremath{\widetilde{E}}}
\def\empt{\varnothing}
\def\ep{\varepsilon}
\def\F{\ensuremath{\mathcal{F}}}
\def\Ft{\ensuremath{\widetilde{\F}}}

\def\G{\ensuremath{\mathcal{G}}}

\def\H{\ensuremath{\mathcal{H}}}

\def\Ind{\ensuremath{\mathbbm{1}}}

\def\lc{\lambda_c}
\def\lcb{\overline{\lambda_c}}
\def\N{\ensuremath{\mathbb{N}}}

\def\P{\ensuremath{\mathbb{P}}}

\def\Pt{\ensuremath{\widetilde{P}}}

\def\R{\ensuremath{\mathbb{R}}}

\def\S{\ensuremath{\mathcal{S}}}
\def\T{\ensuremath{\mathcal{T}}}
\def\Tt{\ensuremath{\widetilde{\T}}}

\def\tree{\ensuremath{\mathfrak{t}}}

\def\Z{\ensuremath{\mathbb{Z}}}

\def\bfr{\mathfrak{b}}
\def\sfr{\mathfrak{s}}
\def\Otilde{\widetilde{O}}

\renewcommand\Re{\operatorname{Re}}
\renewcommand\Im{\operatorname{Im}}

\def\to{\rightarrow}

\def\tand{\ensuremath{\text{ and }}}
\def\tif{\ensuremath{\text{ if }}}
\def\tas{\ensuremath{\text{ as }}}
\def\ton{\ensuremath{\text{ on }}}

\newcommand{\dd}{\mathrm{d}}

\author{\textsc{Pascal Maillard}\thanks{Laboratoire de Probabilités et Modèles Aléatoires, UMR 7599, Université Paris VI, Case courrier 188, 4 Place Jussieu, 75252 PARIS Cedex 05, France, e-mail: \texttt{pascal.maillard@upmc.fr}}}
\title{The number of absorbed individuals in branching Brownian motion with a barrier}

\begin{document}


\maketitle

{\leftskip=2truecm \rightskip=2truecm \baselineskip=15pt \small

\noindent{\bfseries Summary.} We study supercritical branching Brownian motion on the real line starting at the origin and with constant drift $c$. At the point $x > 0$, we add an absorbing barrier, i.e.\ individuals touching the barrier are instantly killed without producing offspring. It is known that there is a critical drift $c_0$, such that this process becomes extinct almost surely if and only if $c \ge c_0$. In this case, if $Z_x$ denotes the number of individuals absorbed at the barrier, we give an asymptotic for $P(Z_x=n)$ as $n$ goes to infinity. If $c=c_0$ and the reproduction is deterministic, this improves upon results of L. Addario-Berry and N. Broutin \cite{AB2009} and E. A\"{\i}d\'ekon \cite{Aid2009} on a conjecture by David Aldous about the total progeny of a branching random walk. The main technique used in the proofs is analysis of the generating function of $Z_x$ near its singular point $1$, based on classical results on some complex differential equations.

%
%
\bigskip


\bigskip

%
%

}

\section{Introduction}
We define branching Brownian motion as follows. Starting with an initial individual sitting at the origin of the real line, this individual moves according to a 1-dimensional Brownian motion with drift $c$ until an independent exponentially distributed time with rate 1. At that moment it dies and produces $L$ (identical) offspring, $L$ being a random variable taking values in the non-negative integers with $P(L=1) = 0$. Starting from the position at which its parent has died, each child repeats this process, all independently of one another and of their parent. For a rigorous definition, see for example \cite{Cha1991}.

We assume that $m = E[L]-1 \in (0,\infty)$, which means that the process is supercritical. At position $x>0$, we add an {\em absorbing barrier}, i.e.\ individuals hitting the barrier are instantly killed without producing offspring. Kesten proved \cite{Kes1978} that this process becomes extinct almost surely if and only if the drift $c \ge c_0 = \sqrt{2 m}$ (he actually needed $E[L^2] < \infty$ for the ``only if'' part, but we are going to prove that the statement holds in general). A conjecture of David Aldous \cite{Ald0000}, originally stated for the branching random walk, says that the number $N_x$ of individuals that have lived during the lifetime of the process satisfies $E[N_x]<\infty$ and $E[N_x \log^+ N_x] = \infty$ in the critical speed area ($c=c_0$), and $P(N_x > n)\sim K n^{-\gamma}$ in the subcritical speed area ($c > c_0$), with some $K > 0, \gamma > 1$. For the branching random walk, the conjecture of the critical case was proven by Addario-Berry and Broutin \cite{AB2009} for general reproduction laws satisfying a mild integrability assumption. A\"{\i}d\'ekon \cite{Aid2009} refined the results for constant $L$ by showing that there are positive constants $\rho,C_1,C_2$, such that for every $x>0$, we have \[\frac{C_1x\e^{\rho x}}{n (\log n)^2} \le P(N_x > n) \le \frac{C_2x\e^{\rho x}}{n (\log n)^2}\quad\text{for large }n.\]

Assuming $L$ constant has the advantage that $N_x$ is directly related to the number $Z_x$ of individuals {\em absorbed at the barrier} by $N_x-1 = (Z_x-1)(L/(L-1))$, hence it is possible to study $N_x$ through $Z_x$. In this sense, Neveu \cite{Nev1988} had already proven the critical case conjecture for branching Brownian motion since he showed that the process $Z=(Z_x)_{x\ge0}$ is actually a continuous-time Galton--Watson process of finite expectation, but with $E[Z_x\log^+ Z_x]=\infty$ for every $x>0$, if $c=c_0$.

Let $\N = \{1,2,3,\ldots\}$ and $\N_0 = \{0\}\cup\N$. Define the infinitesimal transition rates (see \cite{AN1972}, p.\ 104, Equation (6) or \cite{Har1963}, p.\ 95)\[q_n = \lim_{x\downarrow 0} \frac 1 x P(Z_x = n),\quad n \in \N_0\backslash\{1\}.\] We propose a refinement of Neveu's result:

\begin{theorem}
\label{th_tail}
Assume $c=c_0$. Assume that $E[L(\log L)^{2+\ep}] < \infty$ for some $\ep > 0$. Then we have as $n\to\infty$,
\[\sum_{k=n}^\infty q_k \sim \frac{c_0}{n (\log n)^2}\quad\tand\quad P(Z_x > n) \sim \frac{c_0 x \e^{c_0 x}}{n(\log n)^2} \quad\text{for each $x>0$}.\]
\end{theorem}

The heavy tail of $Z_x$ suggests that its generating function is amenable to singularity analysis in the sense of \cite{FO1990}. This is in fact the case in both the critical and subcritical cases if we impose a stronger condition upon the offspring distribution and leads to the next theorem.

Define $f(s)=E[s^L]$ the generating function of the offspring distribution. Denote by $\delta$ the span of $L-1$, i.e.\ the greatest positive integer, such that $L-1$ is concentrated on $\delta\Z$. Let $\lc \le \lcb$ be the two roots of the quadratic equation $\lambda^2 - 2c\lambda + c_0^2 = 0$ and denote by $d=\frac{\lcb}{\lc}$ the ratio of the two roots. Note that $c = c_0$ if and only if $\lc = \lcb$ if and only if $d=1$.

\begin{theorem}
\label{th_density}
Assume that the law of $L$ admits exponential moments, i.e. that the radius of convergence of the power series $E[s^L]$ is greater than 1.
\begin{itemize}[label=$-$]
\item In the critical speed area $(c = c_0)$, as $n\to\infty$,
\[q_{\delta n+1} \sim \frac{c_0}{\delta  n^2 (\log n)^2}\quad \tand \quad P(Z_x = \delta n + 1) \sim  \frac{c_0 x \e^{c_0x}}{\delta  n^2 (\log n)^2} \quad\text{for each $x>0$}.\]
\item In the subcritical speed area $(c > c_0)$ there exists a constant $K = K(c,f) > 0$, such that, as $n\to\infty$,
\[q_{\delta n+1} \sim \frac{K}{n^{d+1}}\quad\tand \quad P(Z_x = \delta n + 1) \sim \frac{\e^{\lcb x} - \e^{\lc x}}{\lcb-\lc} \frac{K}{n^{d+1}} \quad\text{for each $x>0$}.\]
\end{itemize}
Furthermore, $q_{\delta n + k} = P(Z_x = \delta n+ k) = 0$ for all $n\in\Z$ and $k\in\{2,\ldots,\delta\}$.
\end{theorem}

\begin{remark}
The idea of using singularity analysis for the study of $Z_x$ comes from Robin Pemantle's (unfinished) manuscript \cite{Pem1999} about branching random walks with Bernoulli reproduction.
\end{remark}
\begin{remark}
Since the coefficients of the power series $E[s^L]$ are real and non-negative, Pringsheim's theorem (see e.g.\ \cite{FS2009}, Theorem IV.6, p.\ 240) entails that the assumption in Theorem \ref{th_density} is verified if and only if $f(s)$ is analytic at 1.
\end{remark}
\begin{remark}
Let $\beta > 0$ and $\sigma > 0$. We consider a more general branching Brownian motion with branching rate given by $\beta$ and the drift and variance of the Brownian motion given by $c$ and $\sigma^2$, respectively. Call this process the $(\beta,c,\sigma)$-BBM (the reproduction is still governed by the law of $L$, which is fixed). In this terminology, the process described at the beginning of this section is the $(1,c,1)$-BBM. The $(\beta,c,\sigma)$-BBM can be obtained from $(1,c/(\sigma \sqrt{\beta}),1)$-BBM by rescaling time by a factor $\beta$ and space by a factor $\sigma/\sqrt{\beta}$. Therefore, if we add an absorbing barrier at the point $x>0$, the $(\beta,c,\sigma)$-BBM gets extinct a.s.\ if and only if $c\ge c_0 = \sigma\sqrt{2\beta m}$. Moreover, if we denote by $Z_x^{(\beta,c,\sigma)}$ the number of particles absorbed at $x$, we obtain that
\[(Z_x^{(\beta,c,\sigma)})_{x\ge0} \tand (Z_{x\sqrt{\beta}/\sigma}^{(1,c/(\sigma\sqrt{\beta}),1)})_{x\ge0} \text{ are equal in law.}\]
In particular, if we denote the infinitesimal transition rates of $(Z_x^{(\beta,c,\sigma)})_{x\ge0}$ by $q_n^{(\beta,c,\sigma)}$, for $n\in \N_0\backslash\{1\}$, then we have
\[q^{(\beta,c,\sigma)}_n = \lim_{x\downarrow0} \frac 1 x P\Big(Z_x^{(\beta,c,\sigma)} = n\Big) = \frac {\sqrt{\beta}} \sigma \lim_{x\downarrow0} \frac \sigma {x\sqrt{\beta}} P\Big(Z_{x\sqrt{\beta}/\sigma}^{(1,c/(\sigma\sqrt{\beta}),1)} = n\Big) = \frac {\sqrt{\beta}} \sigma q^{(1,c/(\sigma\sqrt{\beta}),1)}_n.\]
One therefore easily checks {\em that the statements of Theorems \ref{th_tail} and \ref{th_density} are still valid for arbitrary $\beta>0$ and $\sigma>0$, provided that one replaces the constants $c_0,\lc,\lcb,K$ by $c_0/\sigma^2$, $\lc/\sigma^2$, $\lcb/\sigma^2$, $\frac{\sqrt \beta}{\sigma} K(c/(\sigma\sqrt{\beta}),f)$, respectively.}
\end{remark}
\begin{remark}
 After submission of this article, Yang and Ren published an article \cite{Yang2011} which permits to weaken the hypothesis in Theorem \ref{th_tail}: It is enough to assume that $E[L(\log L)^2] <\infty$. In our proof, one needs to replace the reference \cite{Kyp2004} by \cite{Yang2011} and use Theorem B of \cite{Bingham1974} instead of our Lemma \ref{lem_XlogX}, in order to obtain \eqref{eq_phiprime_phi}.
\end{remark}

The content of the paper is organised as follows: In Section \ref{sec_probability} we derive some preliminary results by probabilistic means. In Section \ref{sec_FKPP}, we recall a known relation between $Z_x$ and the so-called Fisher--Kolmo\-gorov--Petrovskii--Piskounov (FKPP) equation. Section \ref{sec_tail} is devoted to the proof of Theorem \ref{th_tail}, which draws on a Tauberian theorem and known asymptotics of travelling wave solutions to the FKPP equation. In Section \ref{sec_preliminaries} we review results about complex differential equations, singularity analysis of generating functions and continuous-time Galton--Watson processes. Those are needed for the proof of Theorem \ref{th_density}, which is done in Section \ref{sec_proofdensity}.

\section{First results by probabilistic methods}
\label{sec_probability}
The goal of this section is to prove
\begin{proposition}
\label{prop_heavytail}
Assume $c > c_0$ and $E[L^2] < \infty$. There exists a constant $C = C(x,c,L) > 0$, such that
\[P(Z_x > n) \ge \frac C {n^d}\quad \text{ for large }n.\]
\end{proposition}
This result is needed to assure that the constant $K$ in Theorem \ref{th_density} is non-zero. It is independent from Sections \ref{sec_FKPP} and \ref{sec_tail} and in particular from Theorem \ref{th_tail}. Its proof is entirely probabilistic and follows closely \cite{Aid2009}.

\subsection{Notation and preliminary remarks}
Our notation borrows from \cite{Kyp2004}. An {\em individual} is an element in the space of Ulam--Harris labels \[U=\bigcup_{n\in\N_0} \N^n,\]which is endowed with the ordering relations $\preceq$ and $\prec$ defined by
\[u \preceq v \iff \exists w\in U: v = uw\quad\tand\quad u\prec v \iff u\preceq v \tand u \ne v.\] The space of Galton--Watson trees is the space of subsets $\tree \subset U$, such that $\empt \in \tree$, $v\in \tree$ if $v\prec u$ and $u\in \tree$ and for every $u$ there is a number $L_u\in\N_0$, such that for all $j\in\N$, $uj\in \tree$ if and only if $j\le L_u$. Thus, $L_u$ is the number of children of the individual $u$.

Branching Brownian motion is defined on the filtered probability space $(\T,\F,(\F_t),P)$. Here, $\T$ is the space of Galton--Watson trees with each individual $u\in \tree$ having a mark $(\zeta_u,X_u)\in \R^+\times D(\R^+,\R\cup\{\Delta\})$, where $\Delta$ is a cemetery symbol and $D(\R^+,\R\cup\{\Delta\})$ denotes the Skorokhod space of cadlag functions from $\R^+$ to $\R\cup\{\Delta\}$. Here, $\zeta_u$ denotes the life length and $X_u(t)$ the position of $u$ at time $t$, or of its ancestor that was alive at time $t$. More precisely, for $v\in\tree$, let $d_v = \sum_{w\preceq v}\zeta_w$ denote the time of death and $b_v = d_v-\zeta_v$ the time of birth of $v$. Then $X_u(t) = \Delta$ for $t\ge d_u$ and if $v\preceq u$ is such that $t\in [b_v,d_v)$, then $X_u(t) = X_v(t)$.

The sigma-field $\F_t$ contains all the information up to time $t$, and $\F = \sigma\left(\bigcup_{t\ge 0} \F_t\right)$.

Let $y,c\in\R$ and $L$ be some random variable taking values in $\N_0\backslash\{1\}$. $P = P^{y,c,L}$ is the unique probability measure, such that, starting with a single individual at the point $y$,
\begin{itemize}[nolistsep, label=$-$]
\item each individual moves according to a Brownian motion with drift $c$ until an independent time $\zeta_u$ following an exponential distribution with parameter 1.
\item At the time $\zeta_u$ the individual dies and leaves $L_u$ offspring at the position where it has died, with $L_u$ being an independent copy of $L$.
\item Each child of $u$ repeats this process, all independently of one another and of the past of the process.
\end{itemize}
Note that often $c$ and $L$ are regarded as fixed and $y$ as variable. In this case, the notation $P^y$ is used. In the same way, expectation with respect to $P$ is denoted by $E$ or $E^y$.

A common technique in branching processes since \cite{LPP1995} is to enhance the space $\T$ by selecting an infinite genealogical line of descent from the ancestor $\empt$, called the {\em spine}. More precisely, if $T\in\T$ and $\tree$ its underlying Galton--Watson tree, then $\xi = (\xi_0,\xi_1,\xi_2,\ldots) \in U^{\N_0}$ is a \emph{spine} of $T$ if $\xi_0=\empt$ and for every $n\in\N_0$, $\xi_{n+1}$ is a child of $\xi_n$ in $\tree$. This gives the space
\[\Tt = \{(T,\xi)\in \T\times U^{\N_0}:\xi \text{ is a spine of }T\}\] of marked trees with spine and the sigma-fields $\Ft$ and $\Ft_t$. Note that if $(T,\xi)\in \Tt$, then $T$ is necessarily infinite.

Assume from now on that $m = E[L] - 1 \in (0,\infty)$. Let $N_t$ be the set of individuals alive at time $t$. Note that every $\Ft_t$-measurable function $f:\Tt\to\R$ admits a representation
\[f(T,\xi) = \sum_{u\in N_t} f_u(T)\Ind_{u\in\xi},\]
where $f_u$ is an $\F_t$-measurable function for every $u\in U$. We can therefore define a measure $\Pt$ on $(\Tt,\Ft,(\Ft_t))$ by
\begin{equation}
\label{eq_manyto1}
\int_{\Tt} f \,\dd\Pt = \e^{-m t} \int_{\T} \sum_{u\in N_t} f_u(T) P(\dd T).
\end{equation}
It is known \cite{Kyp2004} that this definition is sound and that $\Pt$ is actually a probability measure with the following properties:
\begin{itemize}[nolistsep, label=$-$]
\item Under $\Pt$, the individuals on the spine move according to Brownian motion with drift $c$ and die at an accelerated rate $m+1$, independent of the motion.
\item When an individual on the spine dies, it leaves a random number of offspring at the point where it has died, this number following the size-biased distribution of $L$. In other words, let $\widetilde{L}$ be a random variable with $E[f(\widetilde{L})] = E[f(L)L/(m+1)]$ for every positive measurable function $f$. Then the number of offspring is an independent copy of $\widetilde{L}$.
\item Amongst those offspring, the next individual on the spine is chosen uniformly. This individual repeats the behaviour of its parent.
\item The other offspring initiate branching Brownian motions according to the law $P$.
\end{itemize}
Seen as an equation rather than a definition, \eqref{eq_manyto1} also goes by the name of ``many-to-one lemma''.

\subsection{Branching Brownian motion with two barriers}
We recall the notation $P^y$ from the previous subsection for the law of branching Brownian motion started at $y\in\R$ and $E^y$ the expectation with respect to $P^y$. Recall the definition of $\Pt$ and define $\Pt^y$ and $\widetilde{E}^y$ analogously.

Let $a,b\in \R$ such that $y\in (a,b)$. Let $\tau = \tau_{a,b}$ be the (random) set of those individuals whose paths enter $(-\infty,a]\cup[b,\infty)$ and all of whose ancestors' paths have stayed inside $(a,b)$. For $u\in\tau$ we denote by $\tau(u)$ the first exit time from $(a,b)$ by $u$'s path, i.e.\
\[\tau(u) = \inf\{t\ge0: X_u(t) \notin (a,b)\} = \min\{t\ge0: X_u(t) \in \{a,b\}\},\] and set $\tau(u) = \infty$ for $u\notin\tau$. The random set $\tau$ is an (optional) stopping line in the sense of \cite{Cha1991}.

For $u\in\tau$, define $X_u(\tau) = X_u(\tau(u))$. Denote by $Z_{a,b}$ the number of individuals leaving the interval $(a,b)$ at the point $a$, i.e.\ \[Z_{a,b} = \sum_{u\in\tau} \Ind_{X_u(\tau) = a}.\]

\begin{lemma}
\label{lem_2barriers}
Assume $|c| > c_0$ and define $\rho = \sqrt{c^2 - c_0^2}$. Then
\[E^y[Z_{a,b}] = \e^{c(a-y)}\frac{\sinh((b-y)\rho)}{\sinh((b-a)\rho)}.\] If, furthermore, $V = E[L(L-1)] < \infty$, then
\[
\begin{split}
E^y[Z_{a,b}^2] = \frac{2V\e^{c(a-y)}}{\rho \sinh^3((b-a)\rho)}\Big[&\sinh((b-y)\rho)\int_a^y\e^{c(a-r)}\sinh^2((b-r)\rho)\sinh((r-a)\rho)\,\dd r\\
 + &\sinh((y-a)\rho)\int_y^b\e^{c(a-r)}\sinh^3((b-r)\rho)\,\dd r\Big] + E^y[Z_{a,b}].
\end{split}
\]
\end{lemma}
\begin{proof}
On the space $\Tt$ of marked trees with spine, define the random variable $I$ by $I = i$ if $\xi_i \in \tau$ and $I=\infty$ otherwise. For an event $A$ and a random variable $Y$ write $E[Y, A]$ instead of $E[Y \Ind_A]$. Then
\[E^y[Z_{a,b}] = E^y\Big[\sum_{u\in\tau} \Ind_{X_u(\tau)=a}\Big] = \Et^y[\e^{m \tau(\xi_I)}, I < \infty, X_{\xi_I}(\tau) = a]
\] by the many-to-one lemma extended to optional stopping lines (see \cite{BK2004}, Lemma 14.1 for a discrete version). But since the spine follows Brownian motion with drift $c$, we have $I < \infty$, $\Pt$-a.s.\ and the above quantity is therefore equal to
\[W^{y,c}[\e^{m T}, B_T = a],\]
where $W^{y,c}$ is the law of standard Brownian motion with drift $c$ started at $y$, $(B_t)_{t\ge 0}$ the canonical process and $T=T_{a,b}$ the first exit time from $(a,b)$ of $B_t$. By Girsanov's theorem, and recalling that $m=c_0^2/2$, this is equal to
\[
W^y[\e^{c(B_T-y)-\frac 1 2 (c^2 - c_0^2) T}, B_T = a],
\] where $W^y = W^{y,0}$. Evaluating this expression (\cite{BS2002}, p.\ 212, Formula 1.3.0.5) gives the first equality.

For $u\in U$, let $\Theta_u$ be the operator that maps a tree in $\T$ to its sub-tree rooted in $u$. Denote further by $C_u$ the set of $u$'s children, i.e.\ $C_u = \{uk: 1\le k \le L_u\}$. Then note that for each $u\in \tau$ we have
\[Z_{a,b} = 1 + \sum_{v\prec u} \mathop{\sum_{w\in C_v}}_{w\npreceq u} Z_{a,b}\circ \Theta_w,\]
hence
\begin{equation}
\label{eq_decomp}
\begin{split}
E^y[Z_{a,b}^2] &= E^y\Big[\sum_{u\in\tau}\Ind_{X_u(\tau)=a}Z_{a,b}\Big]\\
&= E^y[Z_{a,b}] + \Pt^y\Bigg[\e^{m \tau(\xi_I)} \sum_{v \prec \xi_I} \mathop{\sum_{w\in L_v}}_{w\npreceq \xi_I} Z_{a,b}\circ \Theta_w,\ X_{\xi_I}(\tau) = a\Bigg].
\end{split}
\end{equation}
Define the $\sigma$-algebras
\begin{align*}
\G &= \sigma(X_{\xi_I}(t);t\ge0),\\
\H &= \G \vee \sigma(\zeta_v;v\prec \xi_I),\\
\mathcal{I} &= \H \vee \sigma(\xi, I, (L_v;v\prec \xi_I)),
\end{align*}
such that $\G$ contains the information about the path of the spine up to the individual that quits $(a,b)$ first, $\H$ adds to $\G$ the information about the fission times on the spine and $\mathcal I$ adds to $\H$ the information about the individuals of the spine and the number of their children. Now, conditioning on $\mathcal I$ and using the strong branching property, the second term in the last line of \eqref{eq_decomp} is equal to
\[\Pt^y\Bigg[\e^{m \tau(\xi_I)}\sum_{v \prec \xi_I} (L_v-1) E^{X_v(d_v-)}[Z_{a,b}],\ X_{\xi_I}(\tau) = a\Bigg]\]
(recall that $d_v$ is the time of death of $v$). Conditioning on $\H$ and noting the fact that $L_v$ follows the size-biased law of $L$ for an individual $v$ on the spine, yields
\[\Pt^y\Bigg[\e^{m \tau(\xi_I)}\sum_{v \prec \xi_I} \frac{V}{m+1}E^{X_{\xi_I}(d_v)}[Z_{a,b}],\ X_{\xi_I}(\tau) = a\Bigg].\]
Finally, since under $\Pt$ the fission times on the spine form a Poisson process of intensity $m+1$, conditioning on $\G$ and applying Girsanov's theorem yields
\[
\begin{split}
&W^y\left[\e^{c(B_T-y)-\frac 1 2 \rho^2 T}\int_0^T V E^{B_t}[Z_{a,b}] \,\dd t,B_T = a\right]\\
&\hspace{2cm} = V \e^{c(a-y)} \int_a^b E^r[Z_{a,b}] W^y\left[\e^{-\frac 1 2 \rho^2 T} L_T^r, B_T = a\right] \,\dd r,
\end{split}
\]
where $L_T^r$ is the local time of $(B_t)$ at the time $T$ and the point $r$. The last expression can be evaluated explicitly (\cite{BS2002}, p.\ 215, Formula 1.3.3.8) and gives the desired equality.
\end{proof}

\begin{corollary}
\label{cor_2barriers}
Under the assumptions of Lemma \ref{lem_2barriers}, for each $b>0$ there are positive constants $C_b^{(1)}$, $C_b^{(2)}$, such that as $a\to-\infty$,
\begin{enumerate}[nolistsep, label=\alph*)]
\item $E^0[Z_{a,b}] \sim C_b^{(1)} \e^{(c+\rho)a}$,
\item if $c > c_0$, $E^0[Z_{a,b}^2] \sim C_b^{(2)}\e^{(c+\rho)a}$ and
\item if $c < -c_0$, $E^0[Z_{a,b}^2] \sim C_b^{(2)}\e^{2(c+\rho)a}$.
\end{enumerate}
\end{corollary}

The following result is well known and is only included for completeness. We emphasize that the only moment assumption here is $m = E[L]-1 \in (0,\infty)$. Recall that $Z_x$ denotes the number of particles absorbed at $x$ of a BBM started at the origin. For $|c|\ge c_0$, define $\lc$ to be the smaller root of $\lambda^2 - 2c + c_0^2$, thus $\lc = c - \sqrt{c^2-c_0^2}$.

\begin{lemma}
\label{lem_expectation_Z}
Let $x>0$.
\begin{itemize}[nolistsep, label=$-$]
\item If $|c|\ge c_0$, then $E[Z_x] = \e^{\lc x}$.
\item If $|c| < c_0$, then $E[Z_x] = +\infty$.
\end{itemize}
\end{lemma}
\begin{proof}
We proceed similarly to the first part of Lemma \ref{lem_2barriers}. Define the (optional) stopping line $\tau$ of the individuals whose paths enter $[x,\infty)$ and all of whose ancestors' paths have stayed inside $(-\infty,x)$. Define $I$ as in the proof of Lemma \ref{lem_2barriers}. By the stopping line version of the many-to-one lemma we have
\[E[Z_x] = E[\sum_{u\in\tau} 1] = \Et[\e^{m\tau(\xi_I)}, I<\infty].\] By Girsanov's theorem, this equals
\[W[\e^{cx - \frac 1 2 (c^2 - c_0^2) T_x}, T_x<\infty],\] where $W$ is the law of standard Brownian motion started at 0 and $T_x$ is the first hitting time of $x$. The result now follows from \cite{BS2002}, p.\ 198, Formula 1.2.0.1.
\end{proof}

\subsection{Proof of Proposition \ref{prop_heavytail}}
By hypothesis, $c > c_0$, $E[L^2] < \infty$ and the BBM starts at the origin. Let $x>0$ and let $\tau = \tau_x$ be the stopping line of those individuals hitting the point $x$ for the first time. Then $Z_x = |\tau_x|$.

Let $a < 0$ and $n\in\N$. By the strong branching property,
\[P^0(Z_x > n) \ge P^0(Z_x > n \mid Z_{a,x} \ge 1) P^0(Z_{a,x} \ge 1) \ge P^a(Z_x > n) P^0(Z_{a,x} \ge 1).\] If $P^0_-$ denotes the law of branching Brownian motion started at the point $0$ with drift $-c$, then
\[P^a(Z_x > n) = P^0_-(Z_{a-x} > n) \ge P^0_-(Z_{a-x,1} > n).\] In order to bound this quantity, we choose $a = a_n$ in such a way that $n = \frac 1 2 E^0_-[Z_{a_n-x,1}].$ By Corollary \ref{cor_2barriers} a), c) (applied with drift $-c$) and the Paley--Zygmund inequality, there is then a constant $C_1 > 0$, such that
\[P^0_-(Z_{a_n-x,1} > n) \ge \frac 1 4 \frac{E^0_-[Z_{a_n-x,1}]^2}{E^0_-[Z_{a_n-x,1}^2]} \ge C_1\quad\text{ for large }n.\]
Furthermore, by Corollary \ref{cor_2barriers} a) (applied with drift $-c$), we have
\[\frac 1 2 C^{(1)}_1 \e^{-\lc (a_n-x)} \sim n,\quad \tas n\to\infty,\] and therefore $a_n = -(1/\lc)\log n + O(1)$. Again by the Paley--Zygmund inequality and Corollary \ref{cor_2barriers} a), b) (applied with drift $c$), there exists $C_2 > 0$, such that for large $n$,
\[P^0(Z_{a_n,x} \ge 1) = P^0(Z_{a_n,x} > 0) \ge \frac{E^0[Z_{a_n,x}]^2}{E^0[Z_{a_n,x}^2]} \ge \frac{(C^{(1)}_x)^2}{2C^{(2)}_x} \e^{\lcb a_n} \ge \frac{C_2}{n^d}.\] This proves the proposition with $C=C_1C_2$.

\section{The FKPP equation}
\label{sec_FKPP}
As was already observed by Neveu \cite{Nev1988}, the translational invariance of Brownian motion and the strong branching property immediately imply that  $Z=(Z_x)_{x\ge0}$ is a homogeneous continuous-time Galton--Watson process (for an overview to these processes, see \cite{AN1972}, Chapter\;III or \cite{Har1963}, Chapter\;V). There is therefore an infinitesimal generating function
\begin{equation}
\label{eq_def_a}
a(s) = \alpha\left(\sum_{n=0}^\infty p_ns^n - s\right),\quad \alpha > 0,\ p_1 = 0,
\end{equation}
associated to it. It is a strictly convex function on $[0,1]$, with $a(0) \ge 0$ and $a(1)\le 0$. Its probabilistic interpretation is
\[\alpha = \lim_{x\to 0} \tfrac 1 x P(Z_x \ne 1)\quad\tand\quad p_n = \lim_{x\to0} P(Z_x = n | Z_x \ne 1),\] hence $q_n = \alpha p_n$ for $n\in\N_0\backslash\{1\}$. Note that with no further conditions on $c$ and $L$, the sum $\sum_{n\ge0} p_n$ need not necessarily be $1$, i.e.\ the rate $\alpha p_\infty$, where $p_\infty = 1-\sum_{n\ge0} p_n$, with which the process jumps to $+\infty$, may be positive.

We further define $F_x(s) = E[s^{Z_x}]$, which is linked to $a(s)$ by Kolmogorov's forward and backward equations (\cite{AN1972}, p.\;106 or \cite{Har1963}, p.\;102):
\begin{align}
\label{eq_fwd}
\frac{\partial}{\partial x} F_x(s) &= a(s)\frac{\partial}{\partial s} F_x(s) & &\text{(forward equation)}\\
\label{eq_bwd}
\frac{\partial}{\partial x} F_x(s) &= a[F_x(s)] & &\text{(backward equation)}
\end{align}
The forward equation implies that if $a(1)=0$ and $\phi(x) = E[Z_x] = \frac{\partial}{\partial s} F_x(1-)$, then $\phi'(x) = a'(1)\phi(x)$, whence $E[Z_x] = \e^{a'(1)x}$. On the other hand, if $a(1) < 0$, then the process jumps to $\infty$ with positive rate, hence $E[Z_x] = \infty$ for all $x > 0$.

The next lemma is an extension of a result which is stated, but not proven, in \cite{Nev1988}, Equation (1.1). According to Neveu, it is due to A.\ Joffe. To the knowledge of the author, no proof of this result exists in the current literature, which is why we prove it here.

\begin{lemma}
\label{lem_GW_psi}
Let $(Y_t)_{t\ge0}$ be a homogeneous Galton--Watson process started at 1, which may explode and may jump to $+\infty$ with positive rate. Let $u(s)$ be its infinitesimal generating function and $F_t(s) = E[s^{Y_t}]$. Let $q$ be the smallest zero of $u(s)$ in $[0,1]$.
\begin{enumerate}[nolistsep]
\item If $q < 1$, then there exists $t_-\in\R\cup\{-\infty\}$ and a strictly decreasing smooth function $\psi_-:(t_-,+\infty)\to (q,1)$ with $\lim_{t\to t_-}\psi_-(t) = 1$ and $\lim_{t\to \infty}\psi_-(t) = q$, such that on $(q,1)$ we have $u = \psi_-' \circ \psi_-^{-1}$, $F_t(s) = \psi_-(\psi_-^{-1}(s) + t)$.
\item If $q > 0$, then there exists $t_+\in\R\cup\{-\infty\}$ and a strictly increasing smooth function $\psi_+:(t_+,+\infty)\to (0,q)$ with $\lim_{t\to t_+}\psi_+(t) = 0$ and $\lim_{t\to \infty}\psi_+(t) = q$, such that on $(0,q)$ we have $u = \psi_+' \circ \psi_+^{-1}$, $F_t(s) = \psi_+(\psi_+^{-1}(s) + t)$.
\end{enumerate}
The functions $\psi_-$ and $\psi_+$ are unique up to translation.

Moreover, the following statements are equivalent:
\begin{itemize}[nolistsep, label=$-$]
\item For all $t > 0$, $Y_t < \infty$ a.s.
\item $q = 1$ or $t_- = -\infty$.
\end{itemize}
\end{lemma}

\begin{proof}
We first note that $u(s) > 0$ on $(0,q)$ and $u(s) < 0$ on $(q,1)$, since $u(s)$ is strictly convex, $u(0)\ge0$ and $u(1) \le 0$. Since $F_0(s) = s$, Kolmogorov's forward equation \eqref{eq_fwd} implies that $F_t(s)$ is strictly increasing in $t$ for $s\in(0,q)$ and strictly decreasing in $t$ for $s\in(q,1)$. The backward equation \eqref{eq_bwd} implies that $F_t(s)$ converges to $q$ as $t\to\infty$ for every $s\in[0,1)$. Repeated application of \eqref{eq_bwd} yields that $F_t(s)$ is a smooth function of $t$ for every $s\in[0,1]$.

Now assume that $q < 1$. For $n\in\N$ set $s_n = 1-2^{-n}(1-q)$, such that $q < s_1 < 1$, $s_n < s_{n+1}$ and $s_n \to 1$ as $n\to\infty$. Set $t_1 = 0$ and define $t_n$ recursively by \[t_{n+1} = t_n - t',\quad \text{where $t'>0$ is such that } F_{t'}(s_{n+1}) = s_n.\] Then $(t_n)_{n\in\N}$ is a decreasing sequence and thus has a limit $t_- \in \R\cup\{-\infty\}$. We now define for $t\in (t_-,+\infty)$,
\[\psi_-(t) = F_{t-t_n}(s_n),\quad \tif t \ge t_n.\] The function $\psi_-$ is well defined, since for every $n\in\N$ and $t\ge t_n$,
\[F_{t-t_n}(s_n) = F_{t-t_n}(F_{t_n-t_{n+1}}(s_{n+1})) = F_{t-t_{n+1}}(s_{n+1}),\]
by the branching property. The same argument shows us that if $s\in (q,1)$, $s_n > s$ and $t'>0$ such that $F_{t'}(s_n) = s$, then $F_t(s) = F_{t+t'}(s_n) = \psi_-(t+t'+t_n)$ for all $t\ge 0$. In particular, $\psi_-(t'+t_n) = s$, hence $F_t(s) = \psi_-(\psi_-^{-1}(s) + t)$. The backward equation \eqref{eq_bwd} now gives 
\[u(s) = \frac{\partial}{\partial t}F_t(s)_{\big|t=0} = \psi_-'(\psi_-^{-1}(s)).\] The second part concerning $\psi_+$ is proven completely analogously. Uniqueness up to translation of $\psi_-$ and $\psi_+$ is obvious from the requirement $\psi(\psi^{-1}(s) + t) = F_t(s)$, where $\psi$ is either $\psi_-$ or $\psi_+$.

For the last statement, note that $P(Y_t < \infty) = 1$ for all $t>0$ if and only if $F_t(1-) = 1$ for all $t>0$. But this is the case exactly if $q=1$ or $t_- = -\infty$.
\end{proof}

The following proposition shows that the functions $\psi_-$ and $\psi_+$ corresponding to $(Z_x)_{x\ge0}$ are so-called {\em travelling wave} solutions of a reaction-diffusion equation called the Fisher--Kolmo\-gorov--Petrovskii--Piskounov (FKPP) equation. This should not be regarded as a new result, since Neveu (\cite{Nev1988}, Proposition 3) proved it already for the case $c\ge c_0$ and $L=2$ a.s.\ (dyadic branching). However, his proof relied on a path decomposition result for Brownian motion, whereas we show that it follows from simple renewal argument valid for branching diffusions in general.

Recall that $f(s) = E[s^L]$ denotes the generating function of $L$. Let $q'$ be the unique fixed point of $f$ in $[0,1)$ (which exists, since $f'(1) = m+1 > 1$), and let $q$ be the smallest zero of $a(s)$ in $[0,1]$.

\begin{proposition}
\label{prop_FKPP}
Assume $c\in\R$. The functions $\psi_-$ and $\psi_+$ from Lemma \ref{lem_GW_psi} corresponding to $(Z_x)_{x\ge0}$ are solutions to the following differential equation on $(t_-,+\infty)$ and $(t_+,+\infty)$, respectively.
\begin{equation}
\label{eq_FKPP}
\frac{1}{2} \psi'' - c \psi' = \psi - f\circ\psi.
\end{equation}
Moreover, we have the following three cases:
\begin{enumerate}[nolistsep]
\item If $c\ge c_0$, then $q = q'$, $t_- = -\infty$, $a(1) = 0$, $a'(1) = \lc$, $E[Z_x] = \e^{\lc x}$ for all $x>0$.
\item If $|c|<c_0$, then $q = q'$, $t_- \in\R$, $a(1) < 0$, $a'(1) = 2c$, $P(Z_x = \infty) > 0$ for all $x > 0$.
\item If $c\le -c_0$, then $q = 1$, $a(1) = 0$, $a'(1) = \lc$, $E[Z_x] = \e^{\lc x}$ for all $x>0$.
\end{enumerate}
\end{proposition}

\begin{proof}
Let $s\in (0,1)$ and define the function $\psi_s(x) = F_x(s) = E[s^{Z_x}]$ for $x\ge 0$.
By symmetry, $Z_x$ has the same law as the number of individuals $N$ absorbed at the origin in a branching Brownian motion started at $x$ and with drift $-c$. By a standard renewal argument (Lemma\;\ref{lem_difeq_branching_diffusion}), the function $\psi_s$ is therefore a solution of \eqref{eq_FKPP} on $(0,\infty)$ with $\psi_s(0+) = s$. This proves the first statement, in view of the representation of $F_x$ in terms of $\psi-$ and $\psi_+$ given by Lemma\;\ref{lem_GW_psi}.

Let $s\in(0,1)\backslash\{q\}$ and let $\psi(s) = \psi_-(s)$ if $s>q$ and $\psi(s) = \psi_+(s)$ otherwise. By \eqref{eq_FKPP},
\[a'(s) = \frac{\psi''\circ\psi^{-1}(s)}{\psi'\circ\psi^{-1}(s)} = 2c + 2\frac{\psi\circ\psi^{-1}(s) - f\circ\psi\circ\psi^{-1}(s)}{\psi'\circ\psi^{-1}(s)} = 2c + 2\frac{s - f(s)}{a(s)},\] whence, by convexity,
\begin{equation}
\label{eq_a}
a'(s)a(s) = 2ca(s) + 2(s-f(s)),\quad s\in[0,1].
\end{equation}

Assume $|c| \ge c_0$. By Lemma \ref{lem_expectation_Z}, $E[Z_x] = \e^{\lc x}$, hence $a(1) = 0$ and $a'(1) = \lc$, in particular, $a'(1) > 0$ for $c\ge c_0$ and $a'(1) < 0$ for $c \le -c_0$. By convexity, $q < 1$ for $c\ge c_0$ and $q = 1$ for $c \le -c_0$. The last statement of Lemma \ref{lem_GW_psi} now implies that $t_- = -\infty$ if $c\ge c_0$.

Now assume $|c| < c_0$. By Lemma \ref{lem_expectation_Z}, $E[Z_x] = +\infty$ for all $x>0$, hence either $a(1) < 0$ or $a(1) = 0$ and $a'(1) = +\infty$, in particular, $q < 1$ by convexity. However, if $a(1) = 0$, then by \eqref{eq_a}, $a'(1) = 2c - 2m/a'(1)$, whence the second case cannot occur. Thus, $a(1) < 0$ and $a'(1) = 2c$ by \eqref{eq_a}.

It remains to show that $q=q'$ if $q<1$. Assume $q\ne q'$. Then $a(q') \ne 0$ by the (strict) convexity of $a$ and $a'(q') = 2c$ by \eqref{eq_a}. In particular, $a'(q') \ge a'(1)$, which is a contradiction to $a$ being strictly convex.
\end{proof}

\section{Proof of Theorem \ref{th_tail}}
\label{sec_tail}
We start with the following Abelian-type lemma:
\begin{lemma}
\label{lem_XlogX}
Let $X$ be a random variable concentrated on $\N_0$ and let $\varphi(s) = E[s^X]$ be its generating function. Assume that $E[X(\log^+ X)^\gamma] < \infty$ for some $\gamma > 0$. Then, as $s\to0$,
\[\varphi'(1) - \varphi'(1-s) = O((\log \tfrac 1 s)^{-\gamma})\quad\tand\quad\varphi'(1)s + \varphi(1-s) - 1 = O(s(\log \tfrac 1 s)^{-\gamma}).\]
\end{lemma}
\begin{proof}
Let $s_0 > 0$ be such that the function $s\mapsto s (\log \tfrac{1}{s})^\gamma$ is increasing on $[0,s_0]$.  Let $s \in (0,s_0)$. Then, with $p_k = P(X=k)$,
\[(\varphi'(1) - \varphi'(1-s)) (\log \tfrac{1}{s})^{\gamma} = \sum_{k=1}^\infty k p_k (1-(1-s)^{k-1})(\log \tfrac{1}{s})^{\gamma}.\] If $k \ge s^{-1}$, then $(1-(1-s)^{k-1})(\log \tfrac{1}{s})^{\gamma} \le (\log k)^{\gamma}$. If $\lceil s_0^{-1} \rceil \le k < s^{-1}$, then $s(\log \frac 1 s)^\gamma < \frac 1 k (\log k)^\gamma$ and thus ${(1-(1-s)^{k-1})}(\log \tfrac{1}{s})^{\gamma} < ks (\log \tfrac{1}{s})^{\gamma} \le (\log k)^\gamma.$ Hence,
\[\sum_{k=\lceil s_0^{-1} \rceil}^\infty k p_k (1-(1-s)^{k-1})(\log \tfrac{1}{s})^{\gamma} \le \sum_{k=\lceil s_0^{-1} \rceil}^{\infty} p_k k (\log k)^\gamma \le E[X (\log^+ X)^\gamma].\]
Furthermore, we have for $s\in(0,1)$,
\[\sum_{k=1}^{\lceil s_0^{-1} \rceil} kp_k(1-(1-s)^{k-1})(\log \tfrac{1}{s})^{\gamma} \le (s(\log \tfrac{1}{s})^{\gamma}) \sum_{k=1}^{\lceil s_0^{-1} \rceil} k^2p_k \le C,\]
for some $C > 0$. Collecting these results, we have, for every $s\in(0,1)$,
\[(\varphi'(1) - \varphi'(1-s)) (\log \tfrac{1}{s})^{\gamma} \le C+ E[X (\log^+ X)^\gamma] < \infty,\]
by hypothesis. This  yields the first equality. Setting $g(s) = \varphi'(1)s + \varphi(1-s) - 1$, we note that $g(0) = 0$ and \[g'(s) = \varphi'(1) - \varphi'(1-s) = O((\log \tfrac 1 s)^{-\gamma}),\] by the first equality. Since $(\log \tfrac 1 s)^{-\gamma}$ is slowly varying,
\[g(s) = \int_0^s g'(r)\,\dd r = O(s(\log \tfrac 1 s)^{-\gamma}),\] by standard theorems on the integration of slowly varying functions (see e.g.\ \cite{Fel1971}, Section VIII.9, Theorem 1).
\end{proof}

\begin{proof}[Proof of Theorem \ref{th_tail}]
We have $c = c_0$ by hypothesis. Let $\psi_-$ be the travelling wave from Proposition \ref{prop_FKPP}, which is defined on $\R$, since $t_- = -\infty$. Let $\phi(x) = 1-\psi_-(-x)$, such that $\phi(-\infty) = 1-q$, $\phi(+\infty) =0$ and
\begin{equation}
\label{eq_phi}
\tfrac 1 2 \phi''(x) + c_0 \phi'(x) = f(1-\phi(x)) - (1-\phi(x)),
\end{equation}
by \eqref{eq_FKPP}. Furthermore, $a(1-s) = \phi'(\phi^{-1}(s))$ and $F_x(1-s) = 1-\phi(\phi^{-1}(s) - x)$.

Under the hypothesis $E[L(\log L)^{2+\ep}] < \infty$, it is known \cite{Kyp2004} that there exists $K \in (0,\infty)$, such that $\phi(x) \sim Kx\e^{-c_0 x}$ as $x\to\infty$. Since $a(1) = 0$ and $a'(1) = c_0$ by Proposition \ref{prop_FKPP}, this entails that $\phi'(x) = a(1-\phi(x)) \sim -c_0Kx\e^{-c_0x}$, as $x\to\infty$.

Set $\varphi_1 = \phi'$ and $\varphi_2 = \phi$. By \eqref{eq_phi},
\[\frac{\dd}{\dd x} \begin{pmatrix} \varphi_1(x)\\\varphi_2(x) \end{pmatrix} = \begin{pmatrix} \phi''(x)\\\phi'(x) \end{pmatrix} = \begin{pmatrix} -2c_0\phi'(x) + 2[f(1-\phi(x)) - (1-\phi(x))] \\\phi'(x) \end{pmatrix}.\] Setting $g(s) = c_0^2s + 2[f(1-s)-(1-s)] = 2[f'(1)s + f(1-s)-1]$, this gives
\begin{equation}
\label{eq_system_critical}
\frac{\dd}{\dd x} \begin{pmatrix} \varphi_1(x)\\\varphi_2(x) \end{pmatrix} = M \begin{pmatrix} \varphi_1(x)\\\varphi_2(x) \end{pmatrix} + \begin{pmatrix} g(\varphi_2(x))\\0 \end{pmatrix},\quad \text{with }M = \begin{pmatrix} -2c_0 & -c_0^2\\ 1 & 0\end{pmatrix}.
\end{equation}
The Jordan decomposition of $M$ is given by
\begin{equation}
\label{eq_Jordan_critical}
J = A^{-1}MA = \begin{pmatrix}-c_0 & 1 \\ 0 & -c_0\end{pmatrix},\quad A= \begin{pmatrix}-c_0 & 1-c_0\\ 1&1\end{pmatrix}.
\end{equation}
Setting $\begin{pmatrix} \varphi_1\\ \varphi_2 \end{pmatrix}$ = $A\begin{pmatrix} \xi_1\\ \xi_2\end{pmatrix}$, we get with $\xi = \begin{pmatrix} \xi_1 \\ \xi_2 \end{pmatrix}$:
\[\xi'(x) = J\xi(x) + \begin{pmatrix} -g(\phi(x)) \\ g(\phi(x)) \end{pmatrix},\] which, in integrated form, becomes
\begin{equation}
\label{eq_xi_integrated}
\xi(x) = \e^{xJ}\xi(0) + \e^{xJ}\int_0^x \e^{-yJ} \begin{pmatrix} -g(\phi(y)) \\ g(\phi(y)) \end{pmatrix} \,\dd y.
\end{equation}
Note that
\begin{equation}
\label{eq_exJ}
\e^{xJ} = \begin{pmatrix} \e^{-c_0 x} & x\e^{-c_0 x} \\ 0 & \e^{-c_0 x} \end{pmatrix}.
\end{equation}

By the definition of $\xi$ and the above asymptotic of $\phi$, we have $g(\phi(x)) = O(\e^{-c_0 x}/x^{1+\ep})$, as $x\to\infty$, by Lemma \ref{lem_XlogX} and the hypothesis on $L$. Equations \eqref{eq_xi_integrated} and \eqref{eq_exJ} now imply that
\[
 \xi_2(x) \sim \e^{-c_0 x}\Big(\xi_2(0) + \int_0^\infty \e^{c_0 y} g(\phi(y))\,\dd y\Big),
\]
and
\[
 (\xi_1 + \xi_2)(x) \sim x \e^{-c_0 x}\Big(\xi_2(0) + \int_0^\infty \e^{c_0 y} g(\phi(y))\,\dd y\Big),
\]
and since $\phi = \xi_1 + \xi_2$ and $\xi_2 = \phi' + c_0 \phi$, this gives
\begin{equation}
\label{eq_phiprime_phi}
 (\phi' + c_0 \phi)(x) \sim \phi(x)/x \sim K \e^{-c_0 x}.
\end{equation}
With this information, one can now show by elementary calculus (see Section \ref{sec_appendix_tail}), that
\begin{align}
\label{eq_aprimeprime}
a''(1-s) &\sim \frac{c_0}{s (\log \frac 1 s)^2},\quad\tand\\
\label{eq_Fxprimeprime}
F_x''(1-s) &\sim \frac{c_0 x\e^{-c_0 x}}{s (\log \frac 1 s)^2},\quad\tas s\to 0.
\end{align}
By standard Tauberian theorems (\cite{Fel1971}, Section XIII.5, Theorem 5), \eqref{eq_aprimeprime} implies that \[U(n) = \sum_{k=1}^n k^2 q_k \sim c_0 \frac{n}{(\log n)^2}, \quad \tas n\to \infty.\] By integration by parts, this entails that \[\sum_{k=n}^\infty q_k = \int_{n-}^\infty x^{-2} U(\dd x) \sim c_0 \left(2 \int_n^\infty \frac{1}{x^2 (\log x)^2} \,\dd x - \frac{1}{n (\log n)^2}\right).\] But the last integral is equivalent to $1/(n (\log n)^2)$ (\cite{Fel1971}, Section VIII.9, Theorem 1), which proves the first part of the theorem. The second part is proven analogously, using \eqref{eq_Fxprimeprime} instead.
\end{proof}

\section{Preliminaries for the proof of Theorem \ref{th_density}}
\label{sec_preliminaries}

In light of Proposition \ref{prop_heavytail}, one may suggest that under suitable conditions on $L$ one may extend the proof of Theorem \ref{th_tail} to the subcritical case $c > c_0$ and prove that as $n\to\infty$, $P(Z_x > n) \sim C'n^{-d}$ for some constant $C'$. In order to apply Tauberian theorems, one would then have to establish asymptotics for the $(\lfloor d \rfloor + 1)$-th derivatives of $a(s)$ and $F_x(s)$ as $s\to1$. In trying to do this, one quickly sees that the known asymptotics for the travelling wave ($1-\psi(x) \sim \text{const}\times \e^{-\lc x}$ as $x\to-\infty$, see \cite{Kyp2004}) are not precise enough for this method to work. However, instead of relying on Tauberian theorems, one can analyse the behaviour of the {\em holomorphic} function $a(s)$ near its singular point $1$. This method is widely used in combinatorics at least since the seminal paper by Flajolet and Odlyzko \cite{FO1990} and is the basis for our proof of Theorem \ref{th_density}. Not only does it work in both the critical and subcritical cases, it even yields asymptotics for the density instead of the tail only.
%

In the rest of this section, we will define our notation for the complex analytic part of the proof and review some necessary general complex analytic results.

\subsection{Notation}
\label{sec_notation}
In the course of the paper, we will work in the spaces $\C$ and $\C^2$, endowed with the Euclidean topology. An open connected set is called a {\em region}, a simply connected region containing a point $z_0$ is also called a {\em neighbourhood} of $z_0$. The closure of a set $D$ is denoted by $\overline{D}$, its border by $\partial D$. The disk of radius $r$ around $z_0$ is denoted by $\D(z_0,r) = \{z\in\C: |z-z_0| < r\}$, its closure and border by $\Db(z_0,r)$ and $\partial\D(z_0,r)$, respectively. We further use the abbreviation $\D = \D(0,1)$ for the unit disk. For $0 \le \varphi \le \pi$, $r > 0$ and $x\in\R$, we define
\begin{align*}
G(\varphi,r) &= \{z\in\D(1,r)\backslash\{1\}:|\arg(1-z)|<\pi-\varphi\}, & S_+(\varphi,x) &= [x,\infty) \times (-\varphi,\varphi),\\
\Delta(\varphi,r) &= \{z\in\D(0,1+r)\backslash\{1\}: |\arg(1-z)|<\pi-\varphi\}, & S_-(\varphi,x) &= (-\infty,x] \times (-\varphi,\varphi),\\
H(\varphi,r) &= \{z\in\D(0,r)\backslash\{0\}:|\arg z| < \varphi\}.
\end{align*}
Note that $H(\varphi,r) = 1-G(\pi-\varphi,r)$. Here and during the rest of the paper, $\arg(z)$ and $\log(z)$ are the principal values of argument and logarithm, respectively.

Let $G$ be a region in $\C$, $z_0\in \overline{G}$ and $f$ and $g$ analytic functions in $G$ with $g(z)\ne 0$ for all $z\in G$. We write
\begin{align*}
f(z) = o(g(z))\quad&\iff\quad \forall \ep>0\ \exists\delta>0\ \forall z\in G\cap \D(z_0,\delta):|f(z)|\le\ep |g(z)|,\\
f(z) = O(g(z))\quad&\iff\quad \exists C\ge 0\ \exists\delta>0\ \forall z\in G\cap \D(z_0,\delta):|f(z)|\le C |g(z)|,\\
f(z) = \Otilde(g(z))\quad&\iff\quad \exists K\in\C: f(z) = Kg(z) + o(g(z)),\\
f(z) \sim g(z)\quad&\iff\quad f(z) = g(z) + o(g(z)),
\end{align*}
specifying that the relations hold as $z\to z_0$.
\subsection{Complex differential equations}

In this section, we review some basics about complex differential equations. We start with the fundamental existence and uniqueness theorem (\cite{Bie1965}, p.\;1, \cite{Hil1976}, Theorem 2.2.1, p.\ 45 or \cite{Inc1944}, Section 12.1, p.\ 281).

\begin{fact}
\label{th_cde}
Let $G$ be a region in $\C^2$ and $(w_0,z_0)$ a point in $G$. Let $f:G\to \C$ be analytic in $G$, i.e.\ $f$ is continuous and both partial derivatives exist and are continuous. Then there exists a neighbourhood $U$ of $z_0$ and a unique analytic function $w:U\to\C$, such that
\begin{enumerate}[nolistsep]
\item $w(z_0) = w_0$,
\item $(w(z),z)\in G$ for all $z\in U$ and
\item $w'(z) = f(w(z),z)$ for all $z\in U$.
\end{enumerate}
In other words, the differential equation $w' = f(w,z)$ with initial condition $w(z_0) = w_0$ has exactly one solution $w(z)$ which is analytic at $z_0$.
\end{fact}

The following standard result is a special case of a theorem by Painlev\'e (\cite{Bie1965}, p.\ 11, \cite{Hil1976}, Theorem 3.2.1, p.\ 82 or \cite{Inc1944}, Section 12.3, p.\ 286f).

\begin{fact}
\label{fact_extend_border}
Let $H$ be a region in $\C$ and $w(z)$ analytic in $H$. Let $G$ be a region in $\C^2$, such that $(w(z),z)\in G$ for each $z\in H$ and suppose that there exists an analytic function $f:G\to \C$, such that $w'(z) = f(w(z),z)$ for each $z\in H$. Let $z_0\in\partial H$. Suppose that $w(z)$ is continuous at $z_0$ and that $(w(z_0),z_0)\in G$. Then $z_0$ is a regular point of $w(z)$, i.e.\ $w(z)$ admits an analytic extension at $z_0$.
\end{fact}

Let $[z_1,\ldots,z_k]_n$ denote a power series of the variables $z_1,\ldots,z_k$, converging in a neighbourhood of $(0,\ldots,0)$ and which contains only terms of order $n$ or higher. The complex differential equation
\begin{equation}
\label{eq_briot_bouquet}
zw' = \lambda w + pz + [w,z]_2,\quad\lambda,p\in\C,
\end{equation}
was introduced in 1856 by Briot and Bouquet \cite{BB1856} as an example of a complex differential equation admitting analytic solutions at a singular point of the equation. More precisely, they obtained (\cite{Hil1976}, Theorem 11.1.1, p.\;402):
\begin{fact}
\label{fact_bb_holo}
If $\lambda$ is not a positive integer, then there exists a unique function $w(z)$ which is analytic in a neighbourhood of $z=0$ and which satisfies \eqref{eq_briot_bouquet}. Furthermore, $w(0)=0$.
\end{fact}

The singular solutions to this equation were later investigated by Poincar\'e, Picard and others (for a full bibliography, see \cite{HKM1961}). We are going to need the following result (see \cite{HKM1961}, Paragraph III.9.$2^{\text{o}}$ or \cite{Hil1976}, Theorem 11.1.3, p.\;405, but note that the latter reference is without proof and the statement is slightly incomplete).
\begin{fact}
\label{fact_bb_general}
Assume $\lambda > 0$. There exists a function $\psi(z,u) = \sum_{jk\ge0} p_{jk} z^j u^k$, converging in a neighbourhood of $(0,0)$ and such that $p_{00} = 0$ and $p_{01} = 1$, such that the general solution of \eqref{eq_briot_bouquet} which vanishes at the origin is $w = \psi(z,u)$, with
\begin{itemize}[nolistsep, label=$-$]
\item $u = Cz^\lambda$, if $\lambda \notin \N$,
\item $u = z^\lambda(C+K\log z)$, if $\lambda \in \N$.
\end{itemize}
Here, $C\in\C$ is an arbitrary constant and $K\in\C$ is a fixed constant depending only on the right-hand side of \eqref{eq_briot_bouquet}.
\end{fact}
\begin{remark}
\label{rem_bb_general}
The above statement is slightly imprecise, in that the term \emph{solution} is not defined, i.e.\ what \emph{a priori} knowledge of $w(z)$ (regarding its domain of analyticity, smoothness, behaviour at $z=0$, \ldots) is required in order to guarantee that it admits the representation stated in Fact \ref{fact_bb_general}? Inspecting the proof (as in \cite{HKM1961}, for example) shows that it is actually enough to know that $w(z)$ satisfies \eqref{eq_briot_bouquet} on an interval $(0,\ep)$ of the real line and that $w(0+) = 0$. We briefly explain why:

In order to prove Fact \ref{fact_bb_general}, one shows that there exists a function $\psi$ of the form stated above, such that when changing variables by $w = \psi(z,u)$, the function $u(z)$ \emph{formally} satisfies one of the equations
\[zu' = \lambda u\quad\text{or}\quad zu' = \lambda u + K z^\lambda,\] according to whether $\lambda\notin\N$ or $\lambda \in \N$.

Now suppose that $w(z)$ satisfies the above conditions. By the implicit function theorem (\cite{Hor1973}, Theorem 2.1.2), we can invert $\psi$ to obtain a function $\varphi(w,z) = w + qz + [w,z]_2$, $q\in\C$, such that $\psi(z,\varphi(w,z)) = w$ in a neighbourhood of $(0,0)$. We may thus define $u(z) = \varphi(w(z),z)$ for all $z\in(0,\ep_1)$ for some $\ep_1>0$. Moreover, $u(z)$ now truly satisfies the above equations on $(0,\ep_1)$ and $u(0+) = 0$. Standard theory of ordinary differential equations on the real line now yields that $u$ is necessarily of the form stated in Fact \ref{fact_bb_general}.

We further remark that since $u(z)$ is analytic in the slit plane $\C\backslash(-\infty,0]$ and goes to $0$ as $z\to 0$ in $\C\backslash(-\infty,0]$, there exists an $r>0$, such that $(z,u(z))$ is in the domain of convergence of $\psi(z,u)$ for every $z\in H(\pi,r)$. Hence, every solution $w(z)$ can be analytically extended to $H(\pi,r)$.
\end{remark}

\subsection{Singularity analysis}
\label{sec_singanalysis}
We now summarise results about the singularity analysis of generating functions. The basic references are \cite{FO1990} and \cite{FS2009}, Chapter VI. The results are of two types: those that establish an asymptotic for the coefficients of functions that are explicitly  known, and those that estimate the coefficients of functions which are dominated by another function. We start with the results of the first type:

\begin{fact}
\label{th_singanalysis_explicit}
Let $d\in (1,\infty)\backslash\N$, $k\in\N$, $\gamma\in\Z\backslash\{0\}$, $\delta\in\Z$ and the functions $f_1$, $f_2$ defined by
\[
f_1(z) = (1-z)^d\quad\tand\quad f_2(z) = (1-z)^k \left(\log \frac 1 {1-z}\right)^\gamma \left(\log \log \frac 1 {1-z}\right)^\delta,
\]
for $z\in\C\backslash[1,+\infty)$. Let $(p^{(i)}_n)$ be the coefficients of the Taylor expansion of $f_i$ around the origin, $i=1,2$. Then $(p^{(i)}_n)$ satisfy the following asymptotics as $n\to\infty$:
\[
p^{(1)}_n \sim \frac{K_1}{n^{d+1}} \quad\tand\quad p^{(2)}_n \sim \frac{K_2 (\log n)^{\gamma - 1}(\log \log n)^\delta}{n^{k+1}},
\]
for some non-zero constants $K_1 = K_1(d),K_2 = K_2(k,\gamma,\delta)$. We have $K_2(1,-1,0) = 1$.
\end{fact}
\begin{proof}
For $f_1$, this is Proposition 1 from \cite{FO1990}. For $f_2$ this is Remark 3 at the end of Chapter 3 in the same paper. Note that the additional factors $\frac 1 z$ do not change the nature of the singularities, since $\frac 1 z$ is analytic at 1 (see the footnote on p.\ 385 in \cite{FS2009}). The last statement follows from Remark 3 as well.
\end{proof}

The results of the second type are contained in the next theorem. It is identical to Corollary 4 in \cite{FO1990}. Note that a potential difficulty here is that it requires analytical extension outside the unit disk.
\begin{fact}
\label{th_singanalysis_o}
Let $0<\varphi<\pi/2$, $r > 0$ and $f(z)$ be analytic in $\Delta(\varphi,r).$ Assume that as $z\to 1$ in $\Delta(\varphi,r)$,
\[f(z) = o\left((1-z)^\alpha L\left(\frac 1 {1-z}\right)\right),\quad\text{where }L(u) = (\log u)^\gamma (\log\log u)^\delta,\quad \alpha,\gamma,\delta\in\R.\]
Then the coefficients $(p_n)$ of the Taylor expansion of $f$ around 0 satisfy
\[p_n = o\left(\frac{L(n)}{n^{\alpha+1}}\right),\quad\tas n\to\infty.\]
\end{fact}

\subsection{An equation for continuous-time Galton--Watson processes}
\label{sec_GW}
In this section, let $(Y_t)_{t\ge0}$ be a homogeneous continuous-time Galton--Watson process starting at 1. Let $a(s)$ be its infinitesimal generating function and $F_t(s) = E[s^{Y_t}]$. Assume $a(1) = 0$ and $a'(1) = \lambda \in (0,\infty),$ such that $a(s) = 0$ has a unique root $q$ in $[0,1)$.

The following proposition establishes a relation between the infinitesimal generating function of a Galton--Watson process and its generating function at time $t$. For real $s$, the formulae stated in the proposition are well known, but we will need to use them for complex $s$, which is why we have to include some (complicated) hypotheses to be sure that the functions and integrals appearing in the formulae are well defined.

\begin{proposition}
\label{prop_exp_Ft_a}
Suppose that $a$ and $F_t$ have analytic extensions to some regions $D_a$ and $D_F$. Let $Z_a = \{s\in D_a: a(s) = 0\}$. Let there be simply connected regions $G\subset D_a\backslash Z_a$ and $D\subset G\cap D_F$ with $F_t(D) \subset G$ and $D\cap (0,1) \ne \empt$. Then the following equations hold for all $s\in D$:
\begin{equation}
\label{eq_int_Ft_a}
\int_s^{F_t(s)} \frac{1}{a(r)} \,\dd r = t,
\end{equation}
and
\begin{equation}
\label{eq_exp_Ft_a}
1-F_t(s) = \e^{\lambda t}(1-s)\exp\left(-\int_s^{F_t(s)} f^*(r) \,\dd r\right),
\end{equation}
where $f^*(s)$ is defined for all $s\in D_a\backslash Z_a$ as
\begin{equation}
\label{eq_fstar_def}
f^*(s) = \frac{\lambda}{a(s)} + \frac{1}{1-s},
\end{equation}
and the integrals may be evaluated along any path from $s$ to $F_t(s)$ in $G$.
\end{proposition}
\begin{proof}
For $s\in(0,1)\backslash\{q\}$, equation \eqref{eq_int_Ft_a} follows readily from Kolmogorov's backward equation \eqref{eq_bwd}, when the integral is interpreted as the usual Riemann integral (\cite{AN1972}, p.\;106). Now note that by definition of $G$, both $\frac{1}{a(s)}$ and $f^*$ are analytic in the simply connected region $G$ and therefore possess antiderivatives $g$ and $h$ in $G$. Thus, the functions
\[s\mapsto \int_s^{F_t(s)} \frac{1}{a(r)} \,\dd r = g(F_t(s))-g(s)\quad\tand\quad s\mapsto \int_s^{F_t(s)} f^*(r) \,\dd r = h(F_t(s))-h(s)\]
are analytic in $D$. By the analytic continuation principle, \eqref{eq_int_Ft_a} then holds for every $s\in D$, since $D\cap(0,1)\ne \empt$ by hypothesis. This proves the first equation. For the second equation, note that $-\log (1-s)$ is an antiderivative of $\frac{1}{1-s}$ in $G$, whence the right-hand side of \eqref{eq_exp_Ft_a} equals
\[\e^{\lambda t}(1-s)\exp\left(\log(1-F_t(s)) - \log(1-s) - \lambda \int_s^{F_t(s)} \frac{1}{a(r)}\,\dd r \right) = 1-F_t(s),\]
for all $s\in D$, by \eqref{eq_int_Ft_a}. This gives \eqref{eq_exp_Ft_a}.
\end{proof}

\begin{corollary}
\label{cor_1regular}
If 1 is a regular point of $a(s)$, then it is a regular point for $F_t(s)$ for every $t\ge0$.
\end{corollary}
\begin{proof}
Define $G = \{s\in\D: \Re s > q\}$. Then $G \cap Z_a = \empt$, since $q$ is the only zero of $a$ in $\D$ (every probability generating function $g$ with $g'(1) > 1$ has exactly one fixed point $q$ in $\D$; this can easily be seen by applying Schwarz's lemma to $\tau^{-1}\circ g \circ\tau$, where $\tau$ is the M\"obius transformation of the unit disk that maps $0$ to $q$). Let $s_1\in (q,1)$ be such that $F_t(s) \in G$ for every $s\in H = \{s\in\D:\Re s > s_1\}$. We can then apply Proposition \ref{prop_exp_Ft_a} to conclude that \eqref{eq_exp_Ft_a} holds for every $s\in H$.

Since $a(s)$ is analytic in a neighbourhood $U$ of 1 by hypothesis, it is easy to show that $f^*$ is analytic in $U$ as well. Thus, $f^*$ has an antiderivative $F^*$ in $H \cup U$. We define the function $g(s) = (1-s)\exp(F^*(s))$ on $H\cup U$. Since $g'(1) = -\exp(F^*(1)) \ne 0$, there exists an inverse $g^{-1}$ of $g$ in a neighbourhood $U_1$ of $g(1) = 0$. Let $U_2\subset U$ be a neighbourhood of 1, such that $\e^{\lambda t} g(s)\in U_1$ for every $s\in U_2$. Define the analytic function $\widetilde{F}_t(s) = g^{-1}(\e^{\lambda t} g(s))$ for $s\in U_2$. Then by \eqref{eq_exp_Ft_a}, we have $F_t(s) = \widetilde{F}_t(s)$ for every $s\in H\cap U_2$, hence $\widetilde{F}_t$ is an analytic extension of $F_t$ at 1.
\end{proof}

\begin{corollary}
\label{cor_analyticity_Ft}
Suppose that $a(s)$ has an analytic extension to $G(\varphi_0,r_0)$ for some $0 < \varphi_0 < \pi$ and $r_0 > 0$. Suppose further that there exist $c\in\R$, $\gamma > 1$, such that $a(1-s) = -\lambda s + \lambda c s/\log s + O(s/|\log s|^\gamma)$ as $s\to0$. Then for every $\varphi_0 < \varphi < \pi$ there exists $r>0$, such that $F_t(s)$ can be analytically extended to $G(\varphi,r)$, mapping $G(\varphi,r)$ into $G(\varphi_0,r_0)$.
\end{corollary}
\begin{proof}
Recall that $\lambda > 0$. By hypothesis, we can then assume that $a(s) \ne 0$ in $G(\varphi_0,r_0)$ by choosing $r_0$ small enough. Then $\lambda/a$ has an antiderivative $A$ on $G(\varphi_0,r_0)$. Define $B(s) = A(1-s)$ for $s\in H(\pi-\varphi_0,r_0)$, such that \[B'(s) = \frac{1}{s(1-c/\log s + O(|\log s|^{-\gamma}))} = \frac 1 s + \frac c {s\log s} + O\left(\frac 1 {s |\log s|^{-\min(\gamma,2)}}\right).\] We can therefore apply Lemma \ref{lem_inversion_A} to $B$ and deduce that there exist $\varphi_1\in(\varphi_0,\varphi)$ and $r_1,r \in (0,r_0)$, such that $A$ is injective on $G(\varphi_1,r_1)$ and such that $A(s) + \lambda t\in A(G(\varphi_1,r_1))$ for every $s\in G(\varphi,r)$. Hence, $\widetilde{F}_t(s) = A^{-1}(A(s) + \lambda t)$ is defined and analytic on $G(\varphi,r)$. By \eqref{eq_int_Ft_a}, $\widetilde{F}_t(s) = F_t(s)$ on $G(\varphi,r) \cap \D$, hence $\widetilde{F}_t$ is an analytic extension of $F_t$, mapping $G(\varphi,r)$ into $G(\varphi_1,r_1)\subset G(\varphi_0,r_0)$ by definition.
\end{proof}

\section{Proof of Theorem \ref{th_density}}
\label{sec_proofdensity}
We turn back to branching Brownian motion and to our Galton--Watson process $Z=(Z_x)_{x\ge 0}$ of the number of individuals absorbed at the point $x$. Throughout this section, we place ourselves under the hypotheses of Theorem \ref{th_density}, i.e.\ we assume that $c\ge c_0 = \sqrt{2 m}$ and that the radius of convergence of $f(s) = E[s^L]$ is greater than $1$. The equation $\lambda^2 - 2c\lambda + c_0^2 = 0$ then has the solutions $\lc = c - \sqrt{c^2-c_0^2}$ and $\lcb = c + \sqrt{c^2-c_0^2}$, hence $\lc = \lcb = c_0$ if $c = c_0$ and $\lc < c_0 < \lcb$ otherwise. The ratio $d = \lcb/\lc$ is therefore greater than or equal to one, according to whether $c > c_0$ or $c = c_0$, respectively. Recall further that $\delta\in\N$ denotes the span of $L-1$.

Let $a(s) = \alpha(\sum_{k\ge0} p_ks^k - s)$ be the infinitesimal generating function of $Z$ and let $F_x(s)=E[s^{Z_x}]$. We recall the equation \eqref{eq_a} from Section \ref{sec_FKPP}: For $s\in[0,1]$,
\begin{equation}
\label{eq_a_recalled}
a'(s)a(s) = 2ca(s) + 2(s-f(s)).
\end{equation}
By the analytic continuation principle, this equation is satisfied on the domain of analyticity of $a(s)$, in particular, on $\D$.

We now give a quick overview of the proof. Starting point is the equation \eqref{eq_a_recalled}. We are going to see that this equation is closely related to the Briot--Bouquet equation \eqref{eq_briot_bouquet} with $\lambda = d$. The representation of the solution to this equation given by Fact\;\ref{fact_bb_general} will therefore enable us to derive asymptotics for $a(s)$ near its singular point $s=1$ (Theorem\;\ref{th_asymptotics}). Via the results in Section\;\ref{sec_GW}, we will be able to transfer these to the functions $F_x(s)$ (Corollary\;\ref{cor_asymptotics_Ft}). Finally, the theorems of Flajolet and Odlyzko in Section \ref{sec_singanalysis} yield the asymptotics for $q_n$ and $P(Z_x = n)$.

More specifically, we will see that the main singular term in the expansion of $a(1-s)$ or $F_x(1-s)$ near $s=0$ is $s^d$, if $d\notin\N$ and $s^d \log s$, if $d\in\N$. At first sight, this dichotomy might seem strange, but it becomes evident if one remembers that we expect the coefficients of $F_x(s)$ (i.e.\ the probabilities $P(Z_x = n)$, assume $\delta = 1$) to behave like $1/n^{d+1}$, if $d > 1$ (see Proposition \ref{prop_heavytail}). In light of Fact \ref{th_singanalysis_explicit}, a logarithmic factor must therefore appear if $d$ is a natural number, otherwise $F_x(s)$ would be analytic at 1, in which case its coefficients would decrease at least exponentially.

We start by determining the singular points of $a(s)$ and $F_x(s)$ on the boundary of the unit disk, which is the content of the next three lemmas.

\begin{lemma}
\label{lem_delta}
Let $X$ be a random variable with law $(p_k)_{k\in\N_0}$ and let $x>0$. Then the spans of $X-1$ and of $Z_x-1$ are equal to $\delta$.
\end{lemma}
\begin{proof}
This follows from the fact that the BBM starts with one individual and the number of individuals increases by $l-1$ when an individual gives birth to $l$ children.
\end{proof}

\begin{lemma}
\label{lem_h}
If $\delta = 1$, then $a(s)$ and $(F_x(s))_{x>0}$ are analytic at every $s_0\in\partial\D\backslash\{1\}$. If $\delta \ge 2$, then there exist a function $h(s)$ and a family of functions $(h_x(s))_{x>0}$, all analytic on $\D$, such that
\[a(s) = sh(s^\delta)\quad\tand\quad F_x(s) = sh_x(s^\delta),\]
for every $s\in\D$. Furthermore, $h$ and $(h_x)_{x>0}$ are analytic at every $s_0\in\partial\D\backslash\{1\}$.
\end{lemma}
\begin{proof}
Assume first that $\delta \ge 2$. Define
\[h(s) = \alpha(\sum_n p_{1+\delta n} s^n - 1)\quad\tand\quad h_x(s) = \sum_n P(Z_x = 1+\delta n)s^n.\] By Lemma \ref{lem_delta}, $p_{k+\delta n} = P(Z_x = k+\delta n) = 0$ for every $k\in\{2,\ldots,\delta\}$ and $n\in\Z$, whence $a(s) = sh(s^\delta)$ and $F_x(s) = sh_x(s^\delta)$ for every $s\in\D$.

We now claim that $a$ and $F_x$ are analytic at every $s_0\in\partial\D$ with $s_0^\delta \ne 1$. Note that if $\delta \ge 2$, this implies that $h$ and $h_x$ are analytic at every $s_0\in\partial\D\backslash\{1\}$, since the function $s\mapsto s^\delta$ has an analytic inverse in a neighbourhood of any $s\ne0$.

First note that by \cite{Fel1971}, Lemma\;XV.2.3, p.\;475, we have $|\sum_n p_n s_0^n| < 1$ for every $s_0\in\partial\D$, such that $s_0^\delta \ne 1$, whence $a(s_0)\ne 0$. Now write the differential equation \eqref{eq_a_recalled} in the form
\begin{equation*}
\label{eq_a_with g}
a' = \frac{2ca + 2(s-f(s))}{a} =: g(a,s).
\end{equation*}
Since the radius of convergence of $f$ is greater than 1 by hypothesis, $g$ is analytic at $(a(s_0),s_0)$. Furthermore, $a$ is continuous at $s_0$, since $\sum_np_ns^n$ converges absolutely for every $s\in\Db$. Fact\;\ref{fact_extend_border} now shows that $a$ is analytic at $s_0$.

It remains to show that $F_x$ is analytic at $s_0$. Kolmogorov's forward and backward equations \eqref{eq_fwd} and \eqref{eq_bwd} imply that $a(s)F_x'(s) = a(F_x(s))$ on $[0,1]$, and the analytic continuation principle implies that this holds on $\D$. Now, let $s_0\in\partial\D$, such that $s_0^\delta \ne 1$. Then we have just shown that $a$ is analytic and non-zero at $s_0$. Furthermore, $|F_x(s_0)|<1$, by the above stated lemma in \cite{Fel1971} and Lemma \ref{lem_delta}. Thus, the function $f(w,s) = a(w)/a(s)$ is analytic at $(F_x(s_0),s_0)$, hence we can apply Fact\;\ref{fact_extend_border} again to conclude that $F_x$ is analytic at $s_0$ as well.
\end{proof}

The next lemma ensures that we can ignore certain degenerate cases appearing in the course of the analysis of \eqref{eq_a}. It is the analytic interpretation of the probabilistic results in Section \ref{sec_probability}.

\begin{lemma}
\label{lem_1_is_singularity}
1 is a singular point of $a(s)$. If $c = c_0$, then $a''(1) = +\infty$.
\end{lemma}
\begin{proof}
If $c=c_0$, the second assertion follows from Theorem \ref{th_tail} or from Neveu's result that $E[Z_x\log^+ Z_x] = \infty$ for $x > 0$ (see the remark before Theorem \ref{th_tail}). This implies that 1 is a singular point of $a(s)$. If $c > c_0$, Proposition \ref{prop_heavytail} implies that $E[s^{Z_x}] = \infty$ for every $s > 1$, whence $1$ is a singular point of the generating function $F_x(s)$ by Pringsheim's theorem (\cite{FS2009}, Theorem\;IV.6, p.\ 240). By Corollary \ref{cor_1regular}, it follows that 1 is a singular point of $a(s)$ as well.
\end{proof}

The next theorem is the core of the proof of Theorem \ref{th_density}.

\begin{theorem}
\label{th_asymptotics}
Under the assumptions of Theorem \ref{th_density}, for every $\varphi\in (0,\pi)$ there exists $r>0$, such that $a(s)$ possesses an analytical extension (denoted by $a(s)$ as well) to $G(\varphi,r)$. Moreover, as $1-s\to 1$ in $G(\varphi,r)$, the following holds.
\begin{itemize}[label=$-$]
\item If $d=1$, then
\begin{equation}
\label{eq_asymptotics_a_1}
a(1-s) = -c_0 s + c_0 \frac{s}{\log\frac{1}{s}} - c_0s \frac{\log\log\frac{1}{s}}{(\log\frac{1}{s})^2} + \Otilde\left(\frac{s}{(\log\frac{1}{s})^2}\right).
\end{equation}
\item If $d > 1$, then there is a $K = K(c,f)\in\C\backslash\{0\}$ and a polynomial $P(s)=\sum_{n=2}^{\lfloor d \rfloor} c_n s^n$, such that 
\begin{align}
\label{eq_asymptotics_a_not_N}
&\tif d\notin\N:\quad a(1-s) = -\lc s + P(s) + K s^d + o(s^d),\\
\label{eq_asymptotics_a_N}
&\tif d\in\N:\quad a(1-s) = -\lc s + P(s) + K s^d \log s + o(s^d).
\end{align}
\end{itemize}
\end{theorem}

\begin{proof}[Proof of Theorem \ref{th_asymptotics}]
We set $b(s) = a(1-s)$. By \eqref{eq_a_recalled},
\begin{equation}
\label{eq_b_first}
-b'(s)b(s) = 2cb(s) + 2 (1-s - f(1-s))\quad\ton\D(1,1).
\end{equation} Since $f$ is analytic at 1 by hypothesis, there exists $0 < \ep_1 < 1-q$ and a function $g$ analytic on $\D(0,\ep_1)$ with $g(0) = g'(0) = 0$, such that $f(1-s) = 1-(m+1)s+g(s)$ for $s\in \D(0,\ep_1)$.

As a first step, we analyse \eqref{eq_b_first} for real non-negative $s$. Since $\ep_1 < 1-q$, $b(s) < 0$ on $(0,\ep_1)$, whence we can divide both sides by $b(s)$ to obtain
\begin{equation}
\label{eq_b}
\frac{\dd b}{\dd s} = \frac{-2cb -c_0^2 s + 2 g(s)}{b}\quad\ton(0,\ep_1).
\end{equation}

Introduce the parameter $t(s) = \int_s^{\ep_1} \frac{\dd r}{-b(r)}$, $s\in(0,\ep_1]$, such that $t(\ep_1) = 0$, $t(0+) = +\infty$ and $t(s)$ is strictly decreasing on $(0,\ep_1]$. There exists then an inverse $s(t)$ on $[0,\infty)$, which satisfies $s'(t) = b(s(t))$. Hence, we have
\[\frac{\dd b}{\dd t} = \frac{\dd b}{\dd s}\frac{\dd s}{\dd t} = -2cb(t) - c_0^2 s(t) + 2 g(s(t))\quad\ton(0,\infty),\]
In matrix form, this becomes
\begin{equation}
\label{eq_bs}
\frac{\dd}{\dd t} \begin{pmatrix} b\\s \end{pmatrix} = M \begin{pmatrix} b\\s \end{pmatrix} + \begin{pmatrix} 2 g(s)\\0 \end{pmatrix},\quad M = \begin{pmatrix} -2c & -c_0^2\\ 1 & 0\end{pmatrix},
\end{equation} for $t\in(0,\infty)$. Note that this extends \eqref{eq_system_critical} to the subcritical case. This time, the Jordan decomposition of $M$ is given by
\begin{equation}
\label{eq_Jordan_subcritical}
A^{-1}MA = \begin{pmatrix}-\lcb & 0 \\ 0 & -\lc\end{pmatrix},\quad A = \begin{pmatrix}-\lcb & -\lc \\ 1 & 1\end{pmatrix},\quad\tif c > c_0,
\end{equation}
and by \eqref{eq_Jordan_critical}, if $c = c_0$. Setting 
\begin{equation}
\label{eq_bs_BS}
\begin{pmatrix} b \\ s \end{pmatrix} = A \begin{pmatrix}B\\S\end{pmatrix},
\end{equation}
transforms \eqref{eq_bs} into
\begin{alignat}{3}
\label{eq_BS_subcritical}
\frac{\dd B}{\dd t} &= -\lcb B + [B,S]_2, & \frac{\dd S}{\dd t} &= -\lc S + [B,S]_2, &\quad  \tif c &> c_0,\\
\label{eq_BS_critical}
\frac{\dd B}{\dd t} &= -c_0 B + S + [B,S]_2, &\quad \frac{\dd S}{\dd t} &= -c_0 S + [B,S]_2, & \tif c &= c_0,
\end{alignat}
for $t\in(0,\infty)$. Furthermore, by \eqref{eq_bs_BS}, we have 
\begin{align}
\label{eq_BSs}
s &= B + S,\\
\label{eq_S}
S &= \begin{cases}
(\lcb-\lc)^{-1} (b+\lcb s), & \tif d>1,\\
b + c_0 s, & \tif d=1,
\end{cases}\\
\label{eq_B}
B &= \begin{cases}
(\lc - \lcb)^{-1}(b+\lc s), & \tif d>1,\\
-b + (1 - c_0) s, & \tif d=1.
\end{cases}
\end{align}

From now on, let $\ep_2,\ep_3,\ldots$ be positive numbers that are as small as necessary. By the strict convexity of $b$ and the fact that $b'(0) = -\lc$ by Lemma \ref{lem_expectation_Z}, equation \eqref{eq_S} implies that $S$ is a strictly convex non-negative function of $s$ on $[0,\ep_2)$. This implies that the inverse $s = s(S)$ exists and is non-negative and strictly concave on $[0,\ep_3)$. It follows that $t(S)=t(s(S))$ exists on $[0,\ep_4)$. Equations \eqref{eq_BS_subcritical} and \eqref{eq_BS_critical} then yield for $S\in(0,\ep_4)$,
\begin{alignat}{2}
\label{eq_BS_subcritical_2}
\frac{\dd B}{\dd S} &= \frac{d B + [B,S]_2}{S + [B,S]_2}, &\quad & \tif c > c_0,\\
\label{eq_BS_critical_2}
\frac{\dd B}{\dd S} &= \frac{B - c_0^{-1}S + [B,S]_2}{S + [B,S]_2}, && \tif c = c_0.
\end{alignat}

By \eqref{eq_BSs} and the fact that $s(S)$ is strictly concave, $B$ is a strictly concave function of $S$ as well, hence strictly monotone on $(0,\ep_5)$. We claim that $B(S)^2=o(S)$ as $S\to 0$. For $d>1$, one checks by \eqref{eq_S} that $S(s) \sim s$, as $s\to0$, whence $B(S) = o(S)$, as $S\to 0$, by \eqref{eq_BSs}. If $d=1$, then $b'(0) = -c_0$ by Lemma \ref{lem_expectation_Z} and $b''(0) = +\infty$ by Lemma \ref{lem_1_is_singularity}. Equation \eqref{eq_S} then implies that $S(s)/s^2 \to +\infty$ as $s\to 0$, whence $s(S) = o(\sqrt S)$. The claim now follows by \eqref{eq_BSs}.

Proposition \ref{prop_reduction} now tells us that there exists a function $h(z) = [z]_2$, such that the function $\sfr(S) = S-h(B(S))$ has an inverse $S(\sfr)$ on $(0,\ep_6)$ and $\bfr(\sfr) = B(S(\sfr))$ satisfies the Briot--Bouquet equation
\begin{equation}
\label{eq_bfr_sfr}
\sfr \bfr' = \begin{cases}
d \bfr + [\bfr,\sfr]_2, &\tif d > 1,\\
\bfr - c_0^{-1} \sfr + [\bfr,\sfr]_2, &\tif d=1,
\end{cases}
\end{equation}
on $(0,\ep_6)$. By Fact \ref{fact_bb_general} and Remark \ref{rem_bb_general}, there exists then a function $\psi(z,u) = u+rz+[z,u]_2$, $r\in\C$, such that
$\bfr(\sfr) = \psi(\sfr,u(\sfr))$, where 
\[u(z) = Cz^d, \tif d\notin \N\quad\tand\quad u(z) = Cz^d\log z, \tif d\in\N,\] for some constant $C = C(c,f) \in\C$ (the form of $u$ in the case $d\in\N$ can be obtained from the one in Fact \ref{fact_bb_general} by changing $\psi$, $C$ and $K$). Moreover, comparing the coefficient of $\sfr$ on both sides of \eqref{eq_bfr_sfr}, we get, if $d>1$, $r = dr$, whence $r=0$ and if $d = 1$: $r + C = r - c_0^{-1}$, whence $C = -c_0^{-1}$.

Assume now $d>1$. Then $\bfr = u(\sfr) + [\sfr,u(\sfr)]_2$. Recall that $B = \bfr$ and $S = \sfr + h(\bfr)$. By \eqref{eq_BSs},
\[s = B+S = \bfr + \sfr + h(\bfr) = \sfr + u(\sfr) + [\sfr,u(\sfr)]_2,\] such that $s'(\sfr) = 1 + o(1)$ and $s(\sfr) = \sfr + [\sfr]_2 + o(\sfr^\gamma)$, as $\sfr\to0$, where $\gamma=(d+\lfloor d \rfloor)/2$, if $d\notin \N$ and $\gamma = d-1/2$, if $d\in\N$. By Lemmas \ref{lem_inversion} and \ref{lem_inversion_formula_d_ge_1}, for every $\varphi_0\in(0,\pi)$ there exists $r_0>0$, such that the inverse $\sfr(s)$ exists and is analytic on $H(\varphi_0,r_0)$ and satisfies
\[\sfr(s) = s + [s]_2 + o(s^\gamma),\quad\tas s\to0.\] This entails that 
\begin{alignat*}{2}
u(\sfr) &= C\sfr^d = C(s+o(s))^d = Cs^d + o(s^d),&&\quad\tif d\notin\N,\\
u(\sfr) &= C\sfr^d\log \sfr = C(s + o(s^{3/2}))^d \log(s+o(s)) = Cs^d\log s + o(s^d),&&\quad\tif d\in\N\backslash\{1\},\\
\sfr^n &= [s]_2 + o(s^{\gamma+1})+o(s^{\gamma^2}) = [s]_2 + o(s^d),&&\quad\text{ for all $n\ge2$}.
\end{alignat*}
It follows that
\[\bfr(s) = \bfr(\sfr(s)) = u(s) + [s]_2 + o(s^d),\quad\tas s\to0.\]We finally get by \eqref{eq_bs_BS}, \[b = -\lcb B - \lc S = -\lc s + (\lc - \lcb)\bfr = -\lc s + (\lc-\lcb) u(s) + [s]_2 + o(s^d),\] which proves \eqref{eq_asymptotics_a_not_N} and \eqref{eq_asymptotics_a_N}.

If $d=1$, recall that $u(z) = c_0^{-1} z\log \frac 1 z$ and $\bfr = u(\sfr) + r\sfr + [\sfr,u(\sfr)]_2$ for some $r\in\C$. By \eqref{eq_BSs}, \[s = B+S = \bfr + \sfr + h(\bfr) = u(\sfr) + (r+1)\sfr + [\sfr,u(\sfr)]_2,\] such that $s'(\sfr) = c_0^{-1}\log(\tfrac 1 \sfr) + O(1)$ and $s(\sfr) = c_0^{-1} \sfr \log \tfrac 1 \sfr + (r+1)\sfr + o(\sfr)$. Lemma \ref{lem_inversion} now implies that for every $\varphi_0\in(0,\pi)$ there exists $r_0>0$, such that the inverse $\sfr(s)$ exists and is analytic on $H(\varphi_0,r_0)$. Now, by \eqref{eq_bs_BS}, \[b = -c_0 s + S = -c_0 s + \sfr + h(\bfr) = -c_0 s + \sfr + O(\sfr^{3/2}).\] Lemma \ref{lem_inversion_formula_d_eq_1} now yields \eqref{eq_asymptotics_a_1}.
\end{proof}

\begin{remark}
The reason why we cannot explicitly determine the constant $K$ in Theorem \ref{th_asymptotics} is that we are analysing \eqref{eq_a} only {\em locally} around the point 1. Since the solution of \eqref{eq_a} with boundary conditions $a(q) = a(1) = 0$ is unique (this follows from the uniqueness of the travelling wave solutions to the FKPP equation), a {\em global} analysis of this equation should be able to exhibit the value of $K$. But it is probably easier to refine the probabilistic arguments of Section \ref{sec_probability}, which already give a lower bound that can be easily made explicit.
\end{remark}

The asymptotics established in Theorem \ref{th_asymptotics} for the infinitesimal generating function can now be readily transferred to the generating functions $F_x(s)$.

\begin{corollary}
\label{cor_asymptotics_Ft}
Under the assumptions of Theorem \ref{th_density}, for every $x>0$ and $\varphi \in (0,\pi)$ there exists $r>0$, such that $F_x(s) = E[s^{Z_x}]$ can be analytically extended to $G(\varphi,r)$. Furthermore, the following holds as $1-s\to1$ in $G(\varphi,r)$.
\begin{itemize}[label=$-$]
\item If $d=1$, then
\begin{equation}
\label{eq_asymptotics_Ft_1}
F_x(1-s) = 1-\e^{c_0x} s + c_0 x \e^{c_0 x} \left(\frac{s}{\log \frac 1 s} - \frac{s \log\log\frac 1 s}{(\log\frac 1 s)^2}\right) + \Otilde\left(\frac{s}{(\log\frac 1 s)^2}\right).
\end{equation}
\item If $d > 1$, then there is a polynomial $P_x(s)=\sum_{n=2}^{\lfloor d \rfloor} c_n s^n$, such that 
\begin{align}
\label{eq_asymptotics_Ft_not_N}
&\tif d \notin \N:\quad F_x(1-s) = 1-\e^{\lc x} s + P_x(s) + K_x d s^d + o(s^d),\\
\label{eq_asymptotics_Ft_N}
&\tif d\in\N:\quad F_x(1-s) = 1-\e^{\lc x} s + P_x(s) + K_x s^d \log s + o(s^d),
\end{align}
where $K_x = K(\e^{\lcb x} - \e^{\lc x})/(\lcb - \lc)$, with $K$ being the constant from Theorem \ref{th_asymptotics}.
\end{itemize}
\end{corollary}

\begin{proof}
Let $0 < \varphi_0 < \varphi$. By Theorem \ref{th_asymptotics}, there exists $r_0>0$, such that $a(s)$ can be analytically extended to $G(\varphi_0,r_0)$ and satisfies the hypothesis of Corollary \ref{cor_analyticity_Ft}. It follows that there exists $r > 0$, such that $F_x(s)$ can be analytically extended to $G(\varphi,r)$ and maps $G(\varphi,r)$ into $G(\varphi_0,r_0)$. Hence, the functions
\[w(s) = 1-F_x(1-s)\quad\tand\quad I(s) = \int_s^{w(s)} f^*(1-r)\,\dd r,\]
where $f^*(s)$ is defined as in \eqref{eq_fstar_def}, are analytic in $H(\pi-\varphi,r)$. In what follows, we always assume that $s\in H(\pi-\varphi,r)$. Appearance of the symbols $\sim,O,\Otilde,o$ means that we let $s$ go to $0$ in $H(\pi-\varphi,r)$.

First of all, we note that by Proposition \ref{prop_exp_Ft_a}, we have
\begin{equation}
\label{eq_w}
w(s) = s\e^{\lc x}\exp(I(s)) = s\e^{\lc x}\Big(1 + I(s) + \sum_{k=2}^\infty \frac{I(s)^k}{k!}\Big).
\end{equation}
Now assume $d > 1$. By Theorem \ref{th_asymptotics}, $a(1-s) = -\lc s + [s]_2 + u(s) + o(s^d)$, where $u(s) = K s^d$ or $u(s) = K s^d \log s$, according to whether $d\notin \N$ or $d\in\N$, respectively. It follows that
\[
\begin{split}
f^*(1-s) &= \frac{\lc}{a(1-s)} + \frac{1}{s} = -\frac{1}{s}\left(1-[s]_1-\tfrac{u(s)}{\lc s} + o(s^{d-1})\right)^{-1} + \frac 1 s\\
&= [s]_0 - \tfrac{u(s)}{\lc s^2} + o(s^{d-2}).
\end{split}
\]
Now, $\int_s^{w(s)} o(r^{d-2})\,\dd r = o(s^{d-1})$, since $w(s) \sim s\e^{\lc x}$ by Lemma \ref{lem_expectation_Z}. Thus,
\begin{equation}
\label{eq_I}
I(s) = \int_s^{w(s)} f^*(1-r)\,\dd r = [w(s),s]_1 - \int_s^{w(s)}\frac{u(r)}{\lc r^2} \,\dd r + o(s^{d-1}).
\end{equation}
Since $\int_s^{w(s)}r^{-2} u(r) \,\dd r = O(s^{d-1}\log s),$ equations \eqref{eq_w} and \eqref{eq_I} now give
\begin{equation}
\label{eq_w_dg1}
w(s) = s\e^{\lc x}\left(1 + [w(s),s]_1 - \int_s^{w(s)}\frac{u(r)}{\lc r^2} \,\dd r + o(s^{d-1})\right).
\end{equation}
If $d\ge2$, we deduce that $w(s) = s\e^{\lc x} + o(s^{3/2})$. Straightforward calculus now shows that
\begin{equation}
\label{eq_int_sing}
\int_s^{w(s)}\frac{u(r)}{\lc r^2} \,\dd r = \frac{K_x}{\e^{\lc x}}\frac{u(s)}{s} + [w(s),s]_1 + o(s^{d-1}),
\end{equation}
and \eqref{eq_w_dg1} and \eqref{eq_int_sing} now yield
\[w(s) = s\e^{\lc x} + [w(s),s]_2 - K_x u(s) + o(s^d).\]
Repeated application of this equation shows that $w(s) = s\e^{\lc x} + [s]_2 - K_x u(s) + o(s^d)$, which yields \eqref{eq_asymptotics_Ft_not_N} and \eqref{eq_asymptotics_Ft_N}.

In the critical case $d=1$, Theorem \ref{th_asymptotics} tells us that
\begin{equation}
\label{eq_fstar_asymptotics_d1}
f^*(1-s) = \frac{1}{s}\left(-\frac{1}{\log\frac 1 s} + \frac{\log\log\frac 1 s}{(\log\frac 1 s)^2} + \Otilde\left(\frac 1 {(\log\frac 1 s)^2}\right)\right).
\end{equation}
Write $\lambda = \lc = c_0$. For our first approximation of $w(s)$, we note that
\[
I(s) \sim - \int_s^{s\e^{\lambda x}} \frac 1 {r\log \frac 1 r} \,\dd r \sim - \frac{1}{\log \frac 1 s}\int_s^{s\e^{\lambda x}} \frac 1 r\,\dd r = -\frac{\lambda x}{\log \frac 1 s},
\]hence, by \eqref{eq_w},
\begin{equation}
\label{eq_w_first}
w(s) = s\e^{\lambda x}\left(1-\frac{\lambda x}{\log \frac 1 s} + o\left(\frac 1 {\log \frac 1 s}\right)\right).
\end{equation}

To obtain a finer approximation, we decompose $I(s)$ into
\[I(s) = \int_s^{s\e^{\lambda x}} f^*(1-r)\,\dd r + \int_{s\e^{\lambda x}}^{w(s)} f^*(1-r)\,\dd r =: I_1(s) + I_2(s).\]
We then have
\[
I_1(s) = -\frac{\lambda x}{\log \frac 1 s} + \lambda x\frac{\log\log \frac 1 s}{(\log \frac 1 s)^2} + \Otilde\left(\frac 1 {(\log \frac 1 s)^2}\right),\]
and, because of \eqref{eq_w_first},
\[
-I_2(s) \sim \int_{s\e^{\lambda x}}^{s\e^{\lambda x}(1+\frac{\lambda x}{\log s})} \frac 1 {r\log \frac 1 r} \,\dd r \sim \frac{\lambda x}{(\log\frac 1 s)^2}.
\]
Plugging this back into \eqref{eq_w} finishes the proof.
\end{proof}

\begin{proof}[Proof of Theorem \ref{th_density}]
Let $x>0$. We want to apply the methods from singularity analysis reviewed in Section \ref{sec_singanalysis} to the functions $a$ and $F_x$, if $\delta = 1$, or the functions $h$ and $h_x$ from Lemma \ref{lem_h}, if $\delta \ge 2$. Let $\varphi\in(0,\pi/2)$. By Theorem \ref{th_asymptotics} and Corollary \ref{cor_asymptotics_Ft}, there exists $r_0>0$, such that $a$ and $F_x$ can be analytically extended to $G(\varphi,r_0)$, which implies that for some $\varphi_1\in(\varphi,\pi/2)$ and $r_1\in(0,r)$, $h$ and $h_x$ can be extended to $G(\varphi_1,r_1)$, as well. Moreover, by Lemma \ref{lem_h}, each of these functions is analytic in a neighbourhood of every point of $C = \{s\in\partial\D: |1-s| \ge r_1/2\}$, which is a compact set. Hence, there exists a finite number of neighbourhoods which cover $C$. It is then easy to show that there exists $r>0$, such that the functions are analytic in $\Delta(\varphi_1,r)$.

If $\delta = 1$, we can then immediately apply Facts \ref{th_singanalysis_explicit} and \ref{th_singanalysis_o}, together with the asymptotics on $a$ and $F_x$ established in Theorem \ref{th_asymptotics} and Corollary \ref{cor_asymptotics_Ft}, to prove Theorem \ref{th_density}.

If $\delta \ge 2$, let $q(s)$ be the inverse of $s\mapsto s^\delta$ in a neighbourhood of 1, then $h(s) = a(q(s))/q(s)$ near 1, by Lemma \ref{lem_h}. But since $q'(1) = 1/\delta$, we have
\[
h(1-s) = a(1-(\frac 1 \delta s + c_2s^2 + c_3s^3 + \cdots))(1+c_1's+c_2's+\cdots),
\]
for some constants $c_n,c_n'$, and so equations \eqref{eq_asymptotics_a_1}, \eqref{eq_asymptotics_a_not_N} and \eqref{eq_asymptotics_a_N} transfer to $h$ with the coefficient of the main singular term divided by $\delta^d$. We can therefore use Facts \ref{th_singanalysis_explicit} and \ref{th_singanalysis_o} for the function $h$ to obtain the asymptotic for $(p_{\delta n+1})_{n\in\N}$ in Theorem \ref{th_density}. In the same way, equations \eqref{eq_asymptotics_Ft_1}, \eqref{eq_asymptotics_Ft_not_N} and \eqref{eq_asymptotics_Ft_N} yield asymptotics for $h_x$, such that we can use again Facts \ref{th_singanalysis_explicit} and \ref{th_singanalysis_o} to prove the second part of Theorem \ref{th_density}.
\end{proof}

\appendix
\section{Appendix}

\subsection{A renewal argument for branching diffusions}
Let $W =(W_t)_{t\ge0}$ be a diffusion on an interval with endpoints $a'\le 0 < a$, such that $\lim_{x\downarrow 0} \P^x[T_0 < t] = 1$ for every $t>0$, where $T_0 = \inf\{t \ge 0: W_t = 0\}$ and $W_0 = x$, $\P^x$-almost surely. For $x\in (0,a)$, and only in the scope of this section, we define $P^x$ to be the law of the branching diffusion starting with a single particle at position $x$ where the particles move according to the diffusion $W$ and branch with rate $\beta$ according to the reproduction law with generating function $f(s)$. Moreover, particles hitting the point $0$ are absorbed at that point. Denote by $Z$ the number of particles absorbed during the lifetime of the process and define $u_s(x) = P^x[s^Z]$ for $s\in[0,1)$ and $x\in (0,a)$.

\begin{lemma}
\label{lem_difeq_branching_diffusion}
Let $s\in[0,1)$ and $\G$ be the generator of the diffusion $W$. Then
\[\G u_s = \beta(u_s-f\circ u_s)\quad\ton (0,a), \quad\text{with }u_s(0+) = s.\]
\end{lemma}
\begin{proof}
The proof proceeds by a renewal argument similar to the one in \cite{McK1975}. As for the BBM, for an individual $u$, we denote by $\zeta_u$ its time of death, $X_u(t)$ its position at time $t$ and $L_u$ the number of $u$'s children. Define the event $A = \{\exists t\in [0,\zeta_\empt): X_\empt(t) = 0\}$. For $s\in[0,1)$ we have by the strong branching property
\[
\begin{split}
u_s(x) = E^x[s^Z] &= s P^x(A) + E^x\left[\left(E^{X_\empt(\zeta_\empt-)}[s^Z]\right)^{L_\empt}, A^c \right]\\
&= s \P^x(T_0 < \xi) + \E^x[f(u_s(W_\xi)), \xi \le T_0],
\end{split}
\] where $W=(W_t)_{t\ge0}$ is a diffusion with generator $\G$ starting at $x$ under $\P^x$, $T_0 = \inf\{t\ge0: W_t = 0\}$ and the random variable $\xi$ is exponentially distributed with rate $\beta$ and independent from $W$. Setting $v(x) = \P^x(T_0 < \xi)$ we get by integration by parts
\[v(x) = \int_0^\infty \beta \e^{-\beta t} \P^x(T_0 < t) \,\dd t = \int_0^\infty \e^{-\beta t} \P^x(T_0 \in \dd t) = \E^x\left[\e^{-\beta T_0}\right],\]
and therefore $\G v = \beta v$ on $(0,a)$ (\cite{BS2002}, Paragraph II.1.10, p.\ 18).

Denote the $\beta$-resolvent of the diffusion by $R_\beta$. By the strong Markov property,
\[\begin{split}
\E^x[f(u_s(W_\xi)), \xi \le T_0] &= \E^x\Big[\int_0^\infty \beta \e^{-\beta t} f(u_s(W_t)) \,\dd t\Big] - \E^x\Big[\int_{T_0}^\infty \beta \e^{-\beta t} f(u_s(W_t)) \,\dd t\Big]\\
 &= \beta R_\beta(f\circ u_s)(x) - \beta \E^x[\e^{-\beta T_0}]R_\beta(f\circ u_s)(0),
 \end{split}\]
hence $u_s = C_{s,\beta}v + \beta R_\beta(f\circ u_s)$, with $C_{s,\beta} = s - \beta R_\beta(f\circ u_s)(0)$. It follows that
\[\G u_s = \beta C_{s,\beta}v + \beta^2 R_\beta(f\circ u_s) - \beta(f\circ u_s) = \beta(u_s - f\circ u_s)\quad\ton (0,a).\]
By the above hypothesis on $W$, $\P^x(T_0<\xi) \to 1$ as $x\downarrow 0$, whence $u_s(0+) = s$.
\end{proof}

\subsection{Addendum to the proof of Theorem \ref{th_tail}}
\label{sec_appendix_tail}
With the notation used in the proof of Theorem \ref{th_tail}, recall that for some constant $K>0$ we have
\[
 (\phi' + c_0 \phi)(x) \sim \phi(x)/x \sim K \e^{-c_0 x},\quad \text{ as }x\to\infty.
\]
In what follows, formulae containing the symbols $\sim$ and $o()$ are meant to hold as $s\downarrow 0$. The above equation yields
\begin{equation}
\label{eq_a_asymptotic}
a(1-s) = \phi'(\phi^{-1}(s)) = -c_0s + (\phi' + c_0\phi)(\phi^{-1}(s)) = -c_0s + \frac{(c_0+o(1))s}{\log \frac 1 s}.
\end{equation}
Now, by \eqref{eq_a}, we have
\[
 a'(1-s)a(1-s) = 2c_0a(1-s) + c_0^2s - g(s),
\]
where we recall that $g(s)$ was defined as $g(s) = 2(f(1-s) - 1 + f'(1)s)$. From the above equation, one gets
\[
 a''(1-s) = -(a(1-s))^{-3}\Big(\big(c_0a(1-s) + c_0^2s-g(s)\big)^2 - g'(s) a(1-s)^2\Big),
\]
and an application of Lemma \ref{lem_XlogX} and \eqref{eq_a_asymptotic} yields \eqref{eq_aprimeprime}.

Kolmogorov's forward and backward equations \eqref{eq_fwd} and \eqref{eq_bwd} give
\[F_x'(s) = \frac{a(F_x(s))}{a(s)},\]
and taking the derivative on both sides of this equation gives
\begin{equation}
\label{eq_Fxpp_asymp}
F_x''(s) = \frac{a(F_x(s))}{a(s)^2}\left(a'(F_x(s))-a'(s)\right).
\end{equation}
By \eqref{eq_aprimeprime} and $F_x'(1) = E[Z_x] = \e^{c_0x}$, we get
\[
a'(F_x(1-s))-a'(1-s) = -\int_s^{1-F_x(1-s)} a''(1-r)\,\dd r \sim -\frac{c_0^2x}{(\log s)^2}.
\]
This equation, together with \eqref{eq_Fxpp_asymp} now yields
\[F_x''(1-s) \sim \frac{-c_0\e^{c_0x}s}{c_0^2s^2} \left(-\frac{c_0^2x}{(\log s)^2}\right),\] which is \eqref{eq_Fxprimeprime}.

\subsection{Reduction to Briot--Bouquet equations}
\def\wfrak{\mathfrak{w}}
\def\zfrak{\mathfrak{z}}
In this section, we show how one can reduce differential equations as those obtained in the proof of Theorem \ref{th_density} to the canonical form \eqref{eq_briot_bouquet}. It is mostly based on pp.\;64 and 65 of \cite{Bie1965}.

\begin{lemma}
\label{lem_holo_solution}
Let $\lambda \in (0,1]$ and $p\in\C$. Then the equation
\begin{equation}
\label{eq_w_holo_solution}
w' = \frac{\lambda w + [w,z]_2}{z + pw + [w,z]_2}.
\end{equation}
has an analytic solution $w(z) = [z]_2$ in a neighbourhood of the origin.
\end{lemma}
\begin{proof}
We choose the {\em ansatz} $w = z\cdot w_1$. This transforms \eqref{eq_w_holo_solution} into
\[zw_1' + w_1 = \frac{\lambda z w_1+z^2[w_1,z]_0}{z+pzw_1+z^2[w_1,z]_0} = \frac{\lambda w_1+z[w_1,z]_0}{1+[w_1,z]_1}.\] Writing the inverse of the denominator as a power series in $w_1$ and $z$, this equals
\[(\lambda w_1+z[w_1,z]_0)(1+[w_1,z]_1) = \lambda w_1 + rz + [w_1,z]_2,\] for some $r\in\C$. This finally yields \[zw_1' = (\lambda-1) w_1 + rz + [w_1,z]_2.\]
Since $\lambda-1$ is not a positive integer, this equation now has an analytic solution $w_1(z) = [z]_1$ by Fact \ref{fact_bb_holo}, whence $w(z) = zw_1(z) = [z]_2$ solves \eqref{eq_w_holo_solution}.
\end{proof}
\begin{remark}
The important point in Lemma \ref{lem_holo_solution} is that the coefficient of $z$ in the numerator of \eqref{eq_w_holo_solution} is 0, which is why $w'(z) = 0$.
\end{remark}

\begin{proposition}
\label{prop_reduction}
Let $\lambda \ge 1$ and $p\in\R$. Suppose $w(z)$ is a strictly monotone real-valued function on $(0,\ep)$, $\ep>0$, with $w(z)^2 = o(z)$ as $z\to 0$ and satisfying
\begin{equation}
\label{eq_nonreduced}
w' = \frac{\lambda w + pz + [w,z]_2}{z + [w,z]_2}\quad \ton (0,\ep).
\end{equation}
Then there exists $h(z) = [z]_2$ and $\ep_1>0$, such that $\zfrak = z- h(w)$ has an inverse $z = z(\zfrak)$ on $(0,\ep_1)$ and such that
\begin{equation}
\label{eq_reduced}
\zfrak \frac{\dd w}{\dd\zfrak} = \lambda w + p\zfrak + [w,\zfrak]_2\quad\ton(0,\ep_1).
\end{equation}
\end{proposition}
\begin{proof}
By hypothesis, $w(z)$ is monotone on $(0,\ep)$ and therefore possesses an inverse $z = z(w)$ on $(0,\delta)$, $\delta>0$, which satisfies
\begin{equation}
\label{eq_inverse_z_w}
\frac{\dd z}{\dd w} = \frac{\lambda^{-1} z + [w,z]_2}{w + p\lambda^{-1} z + [w,z]_2}.
\end{equation}
By Lemma \ref{lem_holo_solution}, there exists then an analytic solution $z = g(w) = [w]_2$ to \eqref{eq_inverse_z_w}, since $\lambda^{-1}\in (0,1]$ by hypothesis.
Setting $\zfrak = z-g(w)$ transforms \eqref{eq_inverse_z_w} into a differential equation, which has $\zfrak = 0$ as a solution, hence it is of the form
\[\frac{\dd \zfrak}{\dd w} = \frac{\lambda^{-1} \zfrak + \zfrak[w,\zfrak]_1}{w + p\lambda^{-1} \zfrak + [w,\zfrak]_2}.\]
We have $\dd\zfrak/\dd z = 1 + g'(w(z))w'(z) = 1 + O(w(z)w'(z))$. By \eqref{eq_nonreduced},
\[w(z)w'(z) = O(w(z)^2/z + w(z)) = o(1),\] by hypothesis. Hence, there exists $\ep_1>0$, such that  $\zfrak(z)$ is strictly increasing on $(0,\ep_1)$ and therefore has an inverse. Thus, $w(\zfrak) = w(z(\zfrak))$ satisfies
\[\frac{\dd w}{\dd \zfrak} = \frac{w + p\lambda^{-1} \zfrak + [w,\zfrak]_2}{\lambda^{-1} \zfrak (1 + [w,\zfrak]_1)}\quad\ton (0,\ep_2),\] for some $\ep_2 > 0$. Expanding $(1+[w,\zfrak]_1)^{-1}$ as a power series at $(w,\zfrak) = (0,0)$ gives \eqref{eq_reduced}.
\end{proof}

\subsection{Inversion of some analytic functions}
The results in this section are needed in the proofs of Corollary \ref{cor_analyticity_Ft} and Theorem \ref{th_density}.

\begin{lemma}
\label{lem_injective_log}
Let $\varphi\in(0,\pi)$, $r>0$ and $h$ be an analytic function on $H(\varphi,r)$ with $h(z) = o(z)$ as $z\to 0$. Then there exists $r_1>0$, such that for all $z_1,z_2\in H(\varphi,r_1)$,
\[\log z_1 - \log z_2 + \int_{z_2}^{z_1} h(z)\,\dd z \ne 0.\]
\end{lemma}
\begin{proof}
Let $z_1,z_2\in H(\varphi,r)$. Write $z_i = a_i \e^{i\varphi_i}$, with $a_i>0$, $\varphi_i\in(-\varphi,\varphi)$, $i=1,2$. Define the paths
\[\gamma_1(t) = a_2 \e^{i(t\varphi_1 + (1-t)\varphi_2)}\quad\tand\quad \gamma_2(t) = (ta_1+(1-t)a_2)\e^{i\varphi_1},\quad t\in[0,1],\] such that their concatenation forms a path from $z_2$ to $z_1$ in $H(\varphi,r)$. Then
\[\int_{\gamma_1} h(s)\,\dd s = |\varphi_1 - \varphi_2|\cdot a_2o(1/a_2)\quad\tand\quad \int_{\gamma_2} h(s)\,\dd s = |\log a_1 - \log a_2| o(1).\] As a consequence,
\[\left|\int_{z_2}^{z_1} h(s)\,\dd s\right| = (|\varphi_1-\varphi_2|+|\log a_1 - \log a_2|)o(1) \le \sqrt 2 |\log z_1 - \log z_2| o(1).\]
This proves the statement.
\end{proof}

\begin{lemma}
\label{lem_inversion}
Let $r>0$ and $\varphi\in(0,\pi]$. Let $g$ and $h$ be analytic functions on $H(\varphi,r)$ with
$g'(z) = 1+o(1)$, $h'(z) = \log \frac 1 z +O(1)$, $g(z) \to 0$ and $h(z) \to 0$ as $z\to 0$ in $H(\varphi,r)$. Then for each $\varphi_0 \in (0,\varphi)$ and $\varphi_1\in(\varphi_0,\varphi)$ there exist $r_0,r_1>0$, such that $g$ and $h$ are injective on $H(\varphi_1,r_1)$ and the images of $H(\varphi_1,r_1)$ by $g$ and $h$ contain $H(\varphi_0,r_0)$.
\end{lemma}
\begin{proof}
By hypothesis, $g(z) = z+o(z)$ as $z\to0$ in $H(\varphi,r)$, whence $\arg g(z) = \arg z + o(1)$. Thus, there exists $r_1>0$, such that $g(H(\varphi_1,r_1)) \subset \C\backslash(-\infty,0]$.

Suppose that there exist $z_1,z_2\in H(\varphi_1,r_1)$, such that $g(z_1) = g(z_2)$. Let $\gamma$ be a path from $z_2$ to $z_1$ in $H(\varphi_1,r_1)$. Then $g\circ\gamma$ is a loop in $\C\backslash(-\infty,0]$, whence
\[0 = \int_{g\circ\gamma}\frac 1 z\,\dd z = \int_\gamma \frac{g'(z)}{g(z)}\,\dd z = \log z_1 - \log z_2 + \int_\gamma o(\tfrac 1 z)\,\dd z.\]
By Lemma \ref{lem_injective_log}, we can choose $r_1$ so small, that this equality cannot hold, whence $g$ is injective on $H(\varphi_1,r_1)$.

Since $g(z)\to0$ and $\arg g(z) = \arg z + o(1)$ as $z\to0$, there exists $r_0>0$, such that $g(\partial H(\varphi_1,r_1))$ encloses $H(\varphi_0,r_0)$. Now, since $g$ is injective on $H(\varphi_1,r_1)$, $H(\varphi_1,r_1)$ and $g(H(\varphi_1,r_1))$ are conformally equivalent,  whence $g(H(\varphi_1,r_1))$ is simply connected. It follows that $g(H(\varphi_1,r_1)) \supset H(\varphi_0,r_0)$.

Exactly the same arguments hold for $h$, since $h(z) = z(\log\tfrac 1 z + O(1))$ by hypothesis, whence $\arg h(z) = \arg z + o(1)$ and $h'(z)/h(z) = 1/z + o(1/z)$ as $z\to 0$.
\end{proof}

\begin{lemma}
\label{lem_inversion_A}
Let $r>0$, $\varphi\in(0,\pi]$ and $t \in \R$. Let $g$ be an analytic function on $H(\varphi,r)$ with
\[g'(z) = \frac 1 z + \frac{c}{z\log z}+O\left(\frac{1}{z|\log z|^\gamma}\right),\quad\tas z\to 0\text{ in }H(\varphi,r),\] for some $c\in\R$ and $\gamma > 1$. Then for each $0 < \varphi_0 < \varphi_1 < \varphi$ there exist $r_0,r_1>0$, such that $g$ is injective on $H(\varphi_1,r_1)$ and $g(z)+t\in g(H(\varphi_1,r_1))$ for every $z\in H(\varphi_0,r_0)$.
\end{lemma}
\begin{proof}
By the hypothesis on $g$, we have for $z_1,z_2\in H(\varphi,r)$, 
\[g(z_1)-g(z_2) = \log z_1 - \log z_2 + \int_{z_2}^{z_1} o(1/z)\,\dd z.\]
By Lemma \ref{lem_injective_log}, there exists therefore $r_1>0$, such that $g$ is injective on  $H(\varphi_1,r_1)$.

Since $1/(x|\log x|^\gamma)$ is integrable near 0, we have
\[g(z) = \log z + c \log(\log \tfrac 1 z) + o(1),\quad\tas z\to 0,\]
where we assume without loss of generalisation that the constant of integration is 0. It follows that $\Re g(z)\to -\infty$ and $\Im g(z) = \arg z + o(1)$ as $z\to 0$, since $c\in\R$. Hence, there exists an $R\in\R$, such that $g(\partial H(\varphi_1,r_1))$ encloses the strip $S = S_-(R,\varphi_1)$. As in the proof of Lemma \ref{lem_inversion}, it follows that $S\subset g(H(\varphi_1,r_1))$. Furthermore, again by the asymptotics of $\Re g$ and $\Im g$, there exists $r_0>0$, such that $g(s)+ t \in S$ for every $s\in G(\varphi_0,r_0)$. This concludes the proof.
\end{proof}

\begin{lemma}
\label{lem_inversion_formula_d_ge_1}
Let $w(z)$ be an analytic function on an open subset of $\C\backslash(-\infty,0]$, such that $w(z)\to 0$ as $z\to0$ and
\[z = w + a_2w^2 + \cdots + a_nw^n + o(w^\gamma),\quad\tas z\to0,\] for some $n\in\N$, $\gamma > n$ and $a_2,\ldots,a_n\in\C$. Then there exist $b_2,\ldots,b_n\in\C$, such that
\[w(z) = z + b_2z^2 + \cdots + b_nz^n + o(z^\gamma),\quad\tas z\to0.\]
\end{lemma}
\begin{proof}
For every $i\in\N$, we have by hypothesis 
\[z^i = w^i + a_{i,i+1}w^{i+1} + \cdots + a_{i,n}w^n + o(w^\gamma),\] for some $a_{i,i+1},\ldots,a_{i,n}\in\C$. For $2\le k\le n$, define recursively (with $b_1 = 1$) \[b_{k} = -(a_{1,k} + b_2 a_{2,k} + \cdots + b_{k-1} a_{k-1,k}).\] Then, $z + b_2 z^2 + \cdots + b^n z^n = w + o(w^\gamma).$ The statement now follows from the fact that $w(z)\sim z$ as $z\to0$ by hypothesis, whence $o(w^\gamma) = o(z^\gamma)$.
\end{proof}

\begin{lemma}
\label{lem_inversion_formula_d_eq_1}
Let $w(z)$ be an analytic function on an open subset of $\C\backslash(-\infty,0]$, such that $w(z)\to 0$ as $z\to0$ and
\[cz = w\log\tfrac 1 w + Cw + o(w),\quad\tas z\to 0,\] for some constants $c>0$, $C\in\C$. Then
\[w = \frac{cz}{\log \tfrac 1 z}\left(1-\frac{\log\log \tfrac 1 z + C-\log c + o(1)}{\log \tfrac 1 z}\right),\quad\tas z\to0.\]
\end{lemma}

\begin{proof}
Set $f(z) = w(z)/z$. By hypothesis, $\log z \sim \log w = \log z + \log f(z)$, whence $\log f(z) = o(\log z)$. Now define $g(z)$ by
\[w(z) = \frac{cz}{\log \tfrac 1 z} g(z),\]
such that $\log g(z) = \log f(z) - \log \log \tfrac 1 z = o(\log z).$
By hypothesis,
\[
cz \sim w\log\tfrac 1 w = \frac{cz}{\log \tfrac 1 z}g(z)\left(\log\log\tfrac 1 z+\log \tfrac 1 z-\log c-\log g(z)\right) \sim czg(z),
\]
whence $g(z)\sim 1$, which implies $\log g(z) = o(1)$. It now follows from the hypothesis that
\[
cz = \frac{cz}{\log \tfrac 1 z}g(z)\Big(\log\log \tfrac 1 z - \log cz + C + o(1)\Big),
\]
whence
\[g(z) = \left(1+\frac{\log\log \tfrac 1 z + C - \log c + o(1)}{\log \tfrac 1 z}\right)^{-1}.\] The statement now follows from the series representation of $(1+z)^{-1}$ at $z=0$.
\end{proof}

\section*{Acknowledgements}
\addcontentsline{toc}{section}{Acknowledgements}
The author is grateful to Elie A\"{\i}d\'ekon for stimulating discussions, to Louis Boutet de Monvel for several useful comments, to Jean Bertoin for useful comments on the proof of Lemma\;\ref{lem_difeq_branching_diffusion} and to Andreas Kyprianou and Yanxia Ren for having brought to my attention the reference \cite{Yang2011}. Furthermore, the author thanks two anonymous referees for their detailed comments, which led to a considerable improvement in the exposition of the paper.

\setlength{\baselineskip}{13pt}
\addcontentsline{toc}{section}{References}
\bibliography{absorbed_particles}

\end{document}